\documentclass[a4paper, 11pt]{amsart}

\pdfoutput=1
\usepackage{latexsym, amsmath, amssymb, amsthm, mathrsfs, bm, mathtools,comment}
\usepackage[mathscr]{eucal}
\usepackage{array}
\usepackage{tikz}
\usetikzlibrary{arrows,shapes,positioning,calc,patterns,cd}
\usepackage{aliascnt}
\usepackage{graphicx} 
\usepackage[T1]{fontenc} 
\usepackage{mathptmx}
\usepackage{microtype}
\usepackage[colorlinks=true,linkcolor=blue,urlcolor=blue]{hyperref}

\usepackage[onehalfspacing]{setspace}

\usepackage[inner=2.4cm,outer=2.4cm, bottom=3.2cm]{geometry}

\allowdisplaybreaks

\usepackage{xcolor, color, soul}

\usepackage{color}
\usepackage[all]{xy}  
\usepackage{enumerate}
\usepackage{mathbbol,mathtools}

\renewcommand{\leq}{\leqslant}
\renewcommand{\geq}{\geqslant}

\newtheorem{theorem}{Theorem}[section]

\newaliascnt{headcor}{headthm}

\aliascntresetthe{headcor}

\newaliascnt{headconj}{headthm}

\aliascntresetthe{headconj}

\newaliascnt{corollary}{theorem}
\newtheorem{corollary}[corollary]{Corollary}
\aliascntresetthe{corollary}

\newaliascnt{claim}{theorem}

\aliascntresetthe{claim}

\newaliascnt{lemma}{theorem}
\newtheorem{lemma}[lemma]{Lemma}
\aliascntresetthe{lemma}

\newaliascnt{conjecture}{theorem}
\newtheorem{conjecture}[conjecture]{Conjecture}
\aliascntresetthe{conjecture}

\newaliascnt{proposition}{theorem}
\newtheorem{proposition}[proposition]{Proposition}
\aliascntresetthe{proposition}

\theoremstyle{definition}
\newaliascnt{definition}{theorem}
\newtheorem{definition}[definition]{Definition}
\aliascntresetthe{definition}

\newaliascnt{notation}{theorem}
\newtheorem{notation}[notation]{Notation}
\aliascntresetthe{notation}

\newaliascnt{example}{theorem}
\newtheorem{example}[example]{Example}
\aliascntresetthe{example}

\newaliascnt{examples}{theorem}

\aliascntresetthe{examples}

\newaliascnt{remark}{theorem}
\newtheorem{remark}[remark]{Remark}
\aliascntresetthe{remark}

\newaliascnt{fact}{theorem}

\aliascntresetthe{fact}

\newaliascnt{question}{theorem}
\newtheorem{question}[question]{Question}
\aliascntresetthe{question}

\newaliascnt{questions}{theorem}

\aliascntresetthe{questions}

\newaliascnt{problem}{theorem}

\aliascntresetthe{problem}

\newaliascnt{construction}{theorem}

\aliascntresetthe{construction}

\newaliascnt{setup}{theorem}

\aliascntresetthe{setup}

\newaliascnt{algorithm}{theorem}

\aliascntresetthe{algorithm}

\newaliascnt{observation}{theorem}

\aliascntresetthe{observation}

\newaliascnt{discussion}{theorem}
\newtheorem{discussion}[discussion]{Discussion}
\aliascntresetthe{discussion}

\newaliascnt{defprop}{theorem}

\aliascntresetthe{defprop}

\def\equationautorefname~#1\null{(#1)\null}
\def\sectionautorefname~#1\null{Section #1\null}
\def\subsectionautorefname~#1\null{\S #1\null}

\def\trdeg{{\rm trdeg}}

\def\tratto{\mbox{\rule{2mm}{.2mm}$\;\!$}}

\newcommand{\undl}[1]{\underline{#1}}

\newcommand{\ini}[1]{\operatorname{in}_{<}(#1)}

\newcommand{\pars}[1]{\left(#1\right)}

\def \height{{\operatorname{ht}}}
\def \depth{{\operatorname{depth\, }}}

\def \Spec{{\operatorname{Spec\, }}}

\def \Proj{{\operatorname{Proj\, }}}

\def\Hom{{\rm Hom}}
\def\Ext{\operatorname{Ext}}

\def \Ker{{\operatorname{Ker\,}}}

\def \core{{\operatorname{core}}}
\def \gradedcore{{\operatorname{gradedcore}}}

\def \indeg{{\operatorname{indeg}}}

\def \trdeg{{\operatorname{trdeg}}}
\def \multdeg{{\operatorname{deg}}}
\def \lcm{{\operatorname{lcm}}}
\def \Sym{{\operatorname{Sym}}}

\DeclareMathOperator{\gr}{gr}
\DeclareMathOperator{\HS}{HS}

\DeclareMathOperator{\Quot}{Quot}

\def \fv{\mathbf v}
\def \fy{\mathbf y}
\def \ll{\lambda}
\def \fw{\mathbf w}
\def \om{\omega}

\def \f1{\mathbf{1}}

\newcommand{\kk}{\mathbb{k}}

\def\xi{x}

\def\ls{\leqslant}
\def\gs{\geqslant}

\def\ll{\lambda}
\def\llb{\underline{\lambda}}
\def\zb{\underline{z}}

\def\fm{\mathfrak{m}}

\def\fp{\mathfrak{p}}

\def \QQ{\mathbb Q}

\def \AA{\mathbb A}
\def \NN{\mathbb N}
\def \ZZ{\mathbb Z}

\def \C{\mathcal C}

\def \R{\mathcal R}

\def \H{\mathcal H}

\def \A{\mathcal A}

\def \F{\mathcal F}

 \def \mono{{\operatorname{mono}}}

    \def\fx{\mathbf{x}}

\def \core{{\operatorname{core}}}

\def \indeg{{\operatorname{indeg}}}
\def\tn{\textnormal}

\def \adj{{\operatorname{{\rm adj}}}}

\begin{document}

\title[ The core of monomial ideals]{The core of monomial ideals}

\author[Louiza Fouli]{Louiza Fouli}
\address{Louiza Fouli \\ Department of Mathematical Sciences  \\ New Mexico State University  \\PO Box 30001\\Las Cruces, NM 88003-8001}
\email{lfouli@nmsu.edu}

\author[Jonathan Monta{\~n}o]{Jonathan Monta{\~n}o$^1$}
\address{Jonathan Monta{\~n}o\\School of Mathematical and Statistical Sciences, Arizona State University, P.O. Box 871804, Tempe, AZ 85287-18041}
\email{montano@asu.edu}
\thanks{$^{1}$ The second author was partially funded by NSF Grant DMS \#2001645/2303605.}

\author[Claudia Polini]{Claudia Polini$^2$}
\address{Claudia Polini\\Department of Mathematics\\University of Notre Dame\\255 Hurley\\ Notre Dame, IN 46556 }
\email{cpolini@nd.edu}
\thanks{$^{2}$ The third author was partially funded by NSF Grant DMS \#2201110.}

\author[Bernd Ulrich]{Bernd Ulrich$^3$}
\address{Bernd Ulrich\\ Department of Mathematics \\ Purdue University \\150 North University Street \\West Lafayette, IN 47907}
\email{bulrich@purdue.edu}
\thanks{$^{3}$ The fourth author was partially funded by NSF Grant DMS \#2201149.}

  \begin{abstract}
The core of an ideal is defined as the intersection of all of its reductions. In this paper we provide an explicit description for the core of a monomial ideal $I$ satisfying certain residual conditions, showing that ${\rm core}(I)$ coincides with the largest monomial ideal contained in a general  reduction of $I$. We prove that the class of lex-segment ideals  satisfies these residual conditions and study the core of lex-segment ideals generated in one degree. For monomial ideals that do not necessarily satisfy the residual conditions and  that are generated in one degree, we conjecture  an explicit formula for the core, and make progress towards this conjecture.
  \end{abstract}


\maketitle

\section{Introduction}

The {\it core} of an ideal $I$ in a Noetherian ring is the intersection
of all reductions of $I$, i.e., all ideals over which $I$ is integral.  Since reductions, even minimal ones, are highly non-unique, one
uses the core to encode information about all of them. The core appears naturally in the
context of Brian\c con-Skoda theorems that compare the integral closure filtration with
the adic filtration of an ideal \cite{LS, HoHu, L,  Lazarsfeld}. It is also related to adjoints and multiplier ideals \cite{L, HSw}, to Kawamata's conjecture on the non-vanishing of sections
of certain line bundles \cite{HyS1, HyS2}, and to the Cayley-Bacharach property of finite sets of points in projective space  \cite{FPU}.  Knowing the core, say of a zero-dimensional ideal
in a local Cohen-Macaulay ring, can be helpful in proofs via reduction to the Artinian case; for
the elements of $I \setminus {\text{core}}(I)$ are exactly those elements in $I$ that
remain non-zero when reducing modulo some general system of parameters inside $I$, see for instance \cite{Engheta, HMMS}. 

Being an apriori infinite intersection of reductions, the core is difficult to
compute. Explicit formulas for the core have been found, but they require strong hypotheses \cite{HSw, CPU0, HyS1, PU1,  HT,  PUV, W, FPU2, FPU,   SmCore, FM, KohlAngela1, KohlAngela2, Cumming, OWY}.  Without such hypotheses, the best one could hope for is that the core is a finite intersection of {\it general} reductions. 
This was proved in the local case assuming fairly weak {\it residual conditions} \cite{CPU}, see \autoref{sec_back} for definitions.  The first main theorem, \autoref{coreGLMR}, in the current
article generalizes this result to the non-local setting, a non-trivial generalization as the 
core is not known to be compatible with localization. In fact, our result shows aposteriori that the core 
does localize in the setting of the theorem, see \autoref{corelocalizes}. 
If in addition $I$ is generated by homogeneous polynomials of the same degree, we also prove that
the core coincides with the {\it graded core}, the intersection of all homogeneous reductions of $I$, 
see \autoref{remCo}; the question of when this equality holds was also considered by Hyry and Smith 
in connection with their work on Kawamata's conjecture \cite{HyS1}. Without a result as in \autoref{coreGLMR}, the core is essentially uncomputable as one does not know how to identify 
the {\it special} reductions needed in the intersection. In this paper we propose a method for finding such reductions
in the case of monomial ideals, see \autoref{mono section}.

With the same weak residual conditions as in \autoref{coreGLMR} we come close to proving a formula for the core in the monomial case,
by expressing the core of a monomial ideal in terms of a single general reduction. This result is based on the fact
that the core of a monomial ideal $I$ is again monomial, and hence contained in the largest monomial ideal $\mono(K)$ contained
in any reduction $K$. When the reduction $K$ is general, it is highly non-monomial and hence $\mono(K)$ is as close to
the core as possible. In \autoref{MainwithGdS2} we prove that in fact $\core(I)=\mono(K)$ if the aforementioned residual conditions are satisfied. This generalizes a result from \cite{PUV} for the case of zero-dimensional monomial ideals.  The mono of any ideal can be computed using an algorithm by Saito, Sturmfels, and Takayama \cite{SST}, and this is implemented in Macaulay2 and can be accessed with  the command {\tt monomialSubideal}  \cite{GS}.

Examples show that the results described above do not hold without any residual conditions, see \autoref{ex1} and  \autoref{ex2}.
In \autoref{mono section} we treat the graded core of monomial ideals that are generated in a single degree but do not 
satisfy any further assumptions. Whereas the graded core is always contained in $\mono(K)$ for $K$ a general reduction,
in \autoref{mainOneDegree} we come up with a monomial ideal $\mathfrak{A}$ contained in the 
graded core, and we conjecture that in fact $\gradedcore(I)= \mathfrak{A}$, see \autoref{conjectureColon}. We
also propose a way  to find the special reductions required in the intersection that gives the graded core, see \autoref{discus_one_degree}. 
These results use, in an essential way, the ideal $J$ generated by $d$ linear combinations of the monomial
generators of $I \subset R= \kk[x_1, \ldots, x_d]$ with new variables $\undl{z}=z_{ij}$ as coefficients. Considering 
$\kk [\undl{z}][x_1, \ldots, x_d]$ as a polynomial ring in the variables $x_1, \ldots, x_d,$ we form 
the ideal $\mono(J)$, which is generated by monomials $m \in \kk[x_1, \ldots, x_d]$ times ideals $C_m \subset\kk [\undl{z}].$ Due to the variation of the ideals $C_m,$ the ideal $\mono(J)$ carries considerably more information  
than $\mono(K)$ for a general reduction $K \subset R,$ which only records the monomials $m$ and 
does not suffice to determine $\gradedcore(I).$ As an application 
we prove that if the ideals $C_m$ are constant up to radical then the graded core is the mono of a general minimal reduction without any residual conditions, see \autoref{thm_constant_c}.

In the last section of the article we focus on the special class of lex-segment  ideals. We first show that these ideals satisfy the residual conditions as in \autoref{coreGLMR}. 
We conjecture that the core of a lex-segment ideal $I$ generated in a single degree is equal to $I$  times a certain power of the maximal homogeneous ideal, see \autoref{conjectureLex}. We prove one inclusion in full generality and establish the conjecture for a large number of cases, see \autoref{mainCore} and \autoref{lex conj summary}. We also show that the core of $I$ is contained in the adjoint of  $I^{g}$, where $g=\height(I)$.   The connection between cores and adjoints is particularly
attractive in the context of monomial ideals, since there is an
explicit combinatorial description for adjoints in terms of Newton
polyhedra \cite{Howald}, a description that is lacking for cores, even in
the zero-dimensional case.

\section{Background}\label{sec_back}

In this section we provide some background information and fix notations needed in the rest of the article, including the residual conditions  mentioned in the Introduction. For further information we refer to \cite{U1, CEU, HS}.  

Let $R$ be a Cohen-Macaulay ring and $I$ an ideal. A subideal $J\subseteq I$ is a {\it reduction} 
of $I$ if $I$ and $J$ have the same integral closure, or equivalently, if 
\begin{equation}\label{eq_reduction}
I^{n+1}=JI^n\qquad
\text{for}
\qquad
n\gg 0\,.
\end{equation}
The {\it reduction  number of $I$ with respect to $J$}, denoted by $r_J(I)$, is the smallest non-negative integer $n$ for which \autoref{eq_reduction} holds true.

Suppose either $R$ is local with maximal ideal $\fm$ and residue field $\kk$, or $R$ is positively graded over a field $\kk$ with  maximal homogeneous ideal $\fm$ and $I$ is homogeneous generated in a single degree.    We denote by $\ell(I)$ the {\it analytic spread} of $I$, i.e., the dimension of the {\it special fiber ring} $\F(I)=\oplus_{n\gs 0}\, I^n/\fm \, I^n$. If $\kk$ is an infinite field, then $\ell(I)$ is equal to the minimal number of generators $\mu(J)$ of any minimal reduction $J$ of $I$. Recall that a minimal reduction is a reduction that is minimal with respect to inclusion. The {\it reduction number of $I$} is  $r(I)=\min\{r_J(I)\mid  J \text{ is a minimal reduction of } I \}.$

In \cite{AN}, Artin and Nagata defined the notion of $s$-residual intersection that generalizes the notion of linkage when the linked ideals may not have the same height. 
To be precise, an $R$-ideal $K$ in an arbitrary  Cohen-Macaulay ring $R$ is an {\it $s$-residual intersection of $I$} if  $K=\mathfrak a:I$ for some $s$ generated ideal $\mathfrak a\subsetneq I$ such that $\height (K) \gs s$. We say $K$ is a {\it geometric $s$-residual intersection of $I$} if in addition we have $\height(I+K)> s.$ The ideal $I$ is said to be {\it weakly $s$-residually $S_2$}  if the ring $R/K$  satisfies Serre's condition $S_2$ for every $0\ls i\ls s$ and for every  geometric $i$-residual intersection $K$ of $I$. 
We say that $I$ satisfies $G_{s}$ if $\mu(I_{\fp})\ls \height(\fp)$ for every $\fp\in V(I)$ such that $\height(\fp) \ls  s-1$.

In this article we deal with ideals that satisfy the residual conditions $G_d$ and weakly $(d-2)$-residually $S_2$, where $d={\rm dim}(R)$. Classes of ideals that satisfy these two conditions include ideals of dimension one that are generically complete intersections.  
Moreover, if an ideal $I$  satisfies $G_d$, 
then $I$ is weakly $(d-2)$-residually $S_2$ 
if it is strongly Cohen-Macaulay or, more generally, if
after localizing it  has the sliding depth property, see \cite[Theorem 3.1]{H2} and \cite[Theorem 3.3]{HVV}. 
Examples of strongly Cohen-Macaulay ideals are  Cohen-Macaulay almost complete intersections,  Cohen-Macaulay ideals in a Gorenstein ring generated by  $\height(I)+2$ elements \cite[p. 259]{AH}, and ideals in the linkage class of a complete intersection \cite[Theorem 1.11]{H1}, such as  perfect ideals of height 2 \cite{Aper, Gaeta} and perfect Gorenstein ideals of height 3 \cite{Watanabe}. 
In this article we add to this list by proving that   lex-segment ideals satisfy both residual conditions, see \autoref{ArtinN}.

\smallskip

\section{The core and general  reductions} \label{general reductions}

In this section we define a notion of {\it general reductions} for homogeneous ideals that are not necessarily generated in one degree, and we use this new notion to  show a homogeneous version of  \cite[Theorem 4.5]{CPU} for ideals of maximal analytic spread (see \autoref{coreGLMR}).  We begin by setting up some notation.

\begin{notation}\label{Notation}\label{notation_general}
Let $\kk$ be an infinite field,  $R$ a Noetherian $\kk$-algebra, and $I\subset R$ an ideal.  Fix  a positive integer $s$ and $\undl{f}=f_1,\ldots, f_u$ a generating sequence for $I$. 
Consider an $s\times u$ matrix of variables $[\,\underline{z}\,]=[z_{i,j}]$ with $1\ls i\ls s$ and $1\ls j\ls u$.   
 An ideal $J_{s,\,\undl{z}}:=J_{s,\,\undl{z}}(\undl{f})$ of $R[\undl{z}]$ is said to be  generated by  $s$ {\it generic  elements of $I$ $($with respect to $\underline{f}$$)$}, if it is generated by  the entries of the column vector $[\,\undl{z}\,]\, [\,\undl{f}\,]^T.$ 
 
 For every $\underline{\lambda}=(\lambda_{ij})\in \AA^{s u}_\kk$ we define 
$\pi_{\underline{\lambda}}:R[\undl{z}]\rightarrow  R$
to be the evaluation map given by $z_{i,j}\mapsto \lambda_{i,j}$. 
For a positive integer $n$ 
we say that the $R$-ideals $J_1, \ldots, J_n$ are generated by {\it $s$ general elements $($with respect to $\underline{f}$$)$} of $I$, if $J_i= \pi_{\undl{\ll_i}}(J_{s,\undl{z}})$ and $(\undl{\ll_1}, \ldots, \undl{\ll_n})$  ranges over a Zariski dense open subset of $\AA^{nsu}_\kk$. 
\end{notation}

\begin{remark} \label{nonzerodivisor}
Using \autoref{Notation}, let $A \in GL_{u}(R)$ and consider the $R$-automorphism $\phi$ of $R[\,\undl{z}\,]$ that  sends the matrix $[\,\undl{z}\,]$ to $[\,\undl{z}\,] \, A$. If $[\,\undl{g}\,]^{T}:=A\, [\,\undl{f}\,]^{T}$, then $\phi(J_{s,\,\undl{z}}(\undl{f}))=J_{s,\,\undl{z}}(\undl{g})$.
\end{remark}

The following lemma shows that saturating the ideal $J_{s,\,\undl{z}}$ with respect to $I$ is the same as saturating it with respect to  any non-zero element $f\in I$. 

\begin{lemma}\label{prime}
We use \autoref{Notation}. Let $f\in I$ be a non-zero element. If $R$ is a domain, 
then $$
J_{s,\undl{z}}:_{R[\,\undl{z}\,]} I^{\infty}=J_{s,\undl{z}}:_{R[\,\undl{z}\,]} f^{\infty}$$ 
and this is a prime ideal of height $s$. 
\end{lemma}
\begin{proof} 
We clearly have $J_{s,\undl{z}}: I^{\infty}\subseteq J_{s,\undl{z}}: f^{\infty}$. To prove the equality it suffices to show that the  ideal on the left is a prime ideal of height $s$ and the one on the right has height at most $s$. 
Notice that $\height\left(J_{s,\undl{z}}: f^{\infty} \right)\ls \height \left(\left(J_{s,\undl{z}}: f^{\infty}\right)_{f}\right)= \height\left(\left(J_{s,\undl{z}}\right)_{f}\right) \ls s$. 
The last inequality follows by Krull's Altitude Theorem; notice that $\left(J_{s,\undl{z}}\right)_{f}$ is a proper ideal as 
$f\notin \sqrt {J_{s,\undl{z}}}$. 

Since $I$ contains a non-zerodivisor modulo $J_{s,\,\undl{z}}: I^\infty$, 
a general $\kk$-linear combination of $f_1, \ldots, f_u$ is a non-zerodivisor modulo $J_{s,\,\undl{z}}: I^\infty$.
 Hence there exists $A \in GL_{u}(\kk)$ such that ${\begin{bmatrix}g_1\cdots  g_u\end{bmatrix}}^{T}=A\, {\begin{bmatrix}f_1\cdots  f_u\end{bmatrix}}^{T}$ with  $g_1$ a non-zerodivisor modulo $J_{s,\,\undl{z}}: I^\infty$. 
By \autoref{nonzerodivisor} we may replace $f_1, \ldots, f_u$ by $g_1, \ldots, g_u$ to assume $f_1$ is a non-zerodivisor modulo $J_{s,\,\undl{z}}: I^\infty$. 
Notice that  $$\left(J_{s,\,\undl{z}}: I^\infty\right)_{f_1}= \left(J_{s,\,\undl{z}}\right)_{f_1} = \left(\{z_{i,1}+f_1^{-1}\sum_{j=2}^uz_{i,j }f_j \mid 1\ls i\ls s\}\right)R[\undl{z}]_{f_1},$$ which is a prime ideal of height $s$. Since $f_1$ is a non-zerodivisor modulo $\,J_{s,\,\undl{z}}: I^\infty$, it follows that $\, J_{s,\,\undl{z}}: I^\infty$ is a prime ideal of height $s$.  
\end{proof}

The following lemma is needed in the proof of \autoref{reduc}, which in turn provides a way to construct general reductions of  ideals.

\begin{lemma}\label{heights}
Let $\kk$ be an infinite field,   $R$ a finitely generated  $\kk$-algebra, and  $I$  an ideal. If  $J$  
is an ideal generated by $s$ general elements of $I$, then 
$$\dim\left( R/\left(J:_RI^\infty\right)\right) \ls \dim(R) - s$$ 
and 
 $$\dim\left( R /\left( \left(J:_RI^\infty\right) +I \right) \right)\ls \dim(R) - s-1.$$
\end{lemma}
\begin{proof}
The first inequality is \cite[Lemma 2.2]{FPU2}.  The second inequality follows from the first because $I$ contains a non-zerodivisor modulo  
$\,J:I^\infty$.
\end{proof}

If either the ambient ring $R$ is local or the ideal $I$ is generated by forms of the same degree and $R$ is a positively graded $\kk$-algebra, then $d$ general elements of $I$ generate a reduction, where $d=\dim (R)$.  The following proposition   gives a method to construct finite sets of general reductions for arbitrary ideals in any Noetherian $\kk$-algebra (see \autoref{r general red}). 

\begin{proposition}\label{reduc}
Let $\kk$ be an infinite field,   $R$  a finitely generated  $\kk$-algebra of dimension $d$, and  $I$  an ideal of positive height. 
If $R$ is positively graded and $I$ is generated by forms of the same degree, set $f:=0\in R$. Otherwise,
let $f\in I$ be an element not contained in any minimal prime ideal of  $R$ of dimension $d$. 
If $J$ is an ideal generated by $d$ general elements of $I$, then $J+(f)$ is a reduction of $I$. 
\end{proposition}
\begin{proof}
Set $\,H = J +(f)$ and let $^{^{\tratto}}$ denote the images in $\overline{R}=R/(f)$.  Applying \autoref{heights} to the images of $H$ and $I$ in $\overline{R}$ we observe that $\,\overline{H}:_{\overline{R}}\overline{I}^\infty = \overline{R}$, and hence $H$ and $I$ have the same radical. 
Since $\,HI^{n-1}: I^n\subseteq HI^{n}: I^{n+1}$ for every $n\in \NN$, this sequence of ideals stabilizes for $n\gg 0$. Therefore, it suffices to prove that $H_{\fp}$ is a reduction of $I_{\fp}$ for every $\fp\in V(I)$. 
By \autoref{heights} we have $\pars{J: I^\infty} + I  = R$ and therefore $I_{\fp}^n \subseteq J_{\fp}$ for $n\gg 0$ and every $\fp \in V(I)$. 
Let $\gr_I(R)=\oplus_{n\gs 0}I^n/I^{n+1}$. 
We can proceed as in \cite[Proposition 2.3]{Xie} to show that the images in $I/I^2=\left[\gr_I(R)\right]_1$ 
of the $d$ generators of $J$ form a filter regular sequence with respect to $\gr_I(R)_+=\oplus_{n> 0}I^n/I^{n+1}$. Therefore, $J$ is generated by a superficial sequence  of $I$. 
Then, 
 $J\cap I^{n+1} = JI^{n} $ for $n\gg 0$ by \cite[Lemma 8.5.11]{HS}. Thus, 
$J_\fp$ is a reduction of $I_\fp$ for every $\fp\in V(I)$. Thus, $H_{\fp}$ is a reduction of $I_\fp$, which finishes the proof. 
\end{proof}

\begin{definition}\label{r general red}
With  assumptions and notations as in \autoref{reduc}, we say that $K_1, \ldots, K_n$ are {\it{general reductions}} of $I$, if $K_i=J_i+(f)$, where $J_1, \ldots, J_n$ are $n$ ideals generated by $d$ general elements of $I$ as in \autoref{Notation}.
\end{definition}

\begin{remark}\label{remark_r}
Notice that for each $r>0$ we have $K_i':=J_i+(f^r$) is a reduction of $I$ by \autoref{reduc}. 
Furthermore, for $\fm \in V(I)$ a fixed maximal ideal we have $(J_i)_{\fm}$ is a reduction of $I_{\fm}$ and   $(K_i')_{\fm}=(J_i)_{\fm}$ if  $r> r(I_{\fm})$ \cite[Theorem 8.6.6]{HS}. If in addition $\ell(I_{\fm})=d$, then $(K_i')_{\fm}$ is a minimal reduction of $I_{\fm}$. 
In case $R_{\fm}$ is regular and $\dim(R_\fm)=d$, we also have $(K_i')_{\fm}=(J_i)_{\fm}$ for $r\gs d$ by \cite[Corollary 13.3.4]{HS}. 
We also call the ideals $K_i'$ {\it general reductions} of $I$ for any choice of $r$.
\end{remark}

\begin{lemma} \label{S2 localizes}
Let $R$ be a Cohen-Macaulay ring, $s$ a non-negative  integer, and $I$ an ideal satisfying $G_{s+1}$. Then $I$ is weakly $s$-residually $S_2$ if and only if $I_\fp$ is weakly $s$-residually $S_2\,$ for every $\fp \in V(I)$. 
\end{lemma}
\begin{proof}
The property of being weakly $s$-residually $S_2$ localizes (see  \cite[Lemma 2.1(a)]{CPU}). For the converse, let $i \ls s$, let  $K=J: I$ be a geometric $i$-residual intersection of $I$, and let $\fp \in V(K)$. If $\fp \in V(I)$, then $R_{\fp}/K_{\fp}$ is $S_2$ by our assumption. If $\fp \not\in V(I)$, then $K_{\fp}=J_{\fp}$ is a complete intersection, and therefore $R_{\fp}/K_{\fp}$ is  Cohen-Macaulay. 
\end{proof}

\autoref{coreGLMR} extends   \cite[Theorem 4.5]{CPU} from local rings to finitely generated algebras over a field.  We recall that an ideal $I$ is of {\it linear type} if the natural map between its symmetric algebra and Rees algebra is an isomorphism. If $I$ is of linear type it has no proper reductions and hence $\core(I)=I$.

\begin{theorem}\label{coreGLMR}
Let $\kk$ be an infinite field,  $R$ a Cohen-Macaulay finitely generated  $\kk$-algebra of dimension $d$, and  $I$ an ideal of positive height. 
 Assume that $I$ satisfies $G_d$ and is weakly $(d-2)$-residually $S_{2}$. Let $ f\in I$ be as in \autoref{r general red}.   Then there exist positive integers $n$ and $r$ such that 
  $$\core(I)=K_1\cap\cdots \cap K_n\,,$$ where  $K_i = J_i+(f^r)$ for $1\ls i\ls n$  are general reductions of $I$  as in \autoref{remark_r}. 
  If in addition $R$ is regular, then $r$ can be chosen to be $d$. 
\end{theorem}
\begin{proof}
For every $\fp \in V(I)$ the ideal $I_{\fp}$ is $G_d$ and from \autoref{S2 localizes} it follows that $I_{\fp}$ is  weakly $(d-2)$-residually $S_2$.  
Hence  \cite[Corollary~3.6  (b)]{CEU} and \cite[Theorem 2.3.2]{VasBlowup}  show that $I_{\fp}$ is of linear type for every $\fp \in V(I)$ with $\height(\fp) <d$. Clearly, $I_{\fp}$ is of linear type for every prime $\fp \not\in V(I)$.

Let $\Sym(I)$ and $\R(I)$ be the symmetric algebra and the Rees algebra of $I$, respectively. Consider the following exact sequence $$0\to \A\to \Sym(I)\rightarrow \R(I)\to 0\,.$$ The ideal $\A$ is generated by homogeneous elements of degree at most $e$, for some non-negative integer $e$. 
Therefore, ${\rm{Supp}}_{R}{(\A)}=\bigcup \limits_{i=0}^e {\rm{Supp}}_R(\A_i)$ is a closed subset of ${\rm{Spec}}(R)$. It follows that the set of prime ideals $\fp$ such that $I_{\fp}$ is not of linear type consists only of finitely many maximal ideals, say $\fm_1, \ldots, \fm_t$. 
Notice that $\ell(I_{\fm_i})=d$ for each $1\ls i \ls t$; indeed, if $\ell(I_{\fm_i})<d$, then $I_{\fm_i}$ is generated by $d-1$ elements according to \cite[Lemma 2.1(g)]{CPU} and hence $I_{\fm_i}$ would be  of linear type  by \cite[Corollary 3.6 (b)]{CEU} and \cite[Theorem 2.3.2]{VasBlowup}.

The ideals $I_{\fm_i}$ have analytic spread $d$,  satisfy $G_d$, and are weakly $(d-2)$-residually $S_{2}$, hence are weakly $(d-1)$-residually $S_{2}$ (see \cite[Proposition  3.4 (a)]{CEU}). 
Applying  \cite[Theorem 4.5]{CPU} and \autoref{remark_r} to the finitely many ideals $I_{\fm_i}$, we obtain that  $\core(I_{\fm_i})=(K_1)_{\fm_i} \cap \cdots \cap (K_n)_{\fm_i}$ for some integer $n$, where  $K_1,\ldots, K_n$   are general  reductions of $I$ with $r=1+\max\{r(I_{\fm_i})\mid 1 \ls i \ls t\}$ or with $r=d$ in case $R$ is  regular. 

We claim that $\core(I)= K_1 \cap \cdots \cap K_n$. Clearly, $\core(I) \subseteq K_1 \cap \cdots \cap K_n$, because $K_1, \ldots, K_n$ are reductions of $I$ by \autoref{reduc}. To show the reverse inclusion, let $K$ be any reduction of $I$. 
We need to show that $K_1 \cap \cdots \cap K_n \subseteq K$, or equivalently $(K_1)_{\fp} \cap \cdots \cap (K_n)_{\fp} \subseteq K_{\fp}$ for every $\fp \in {\rm{Spec}}(R)$. If $\fp \not\in \{\fm_1, \ldots, \fm_t\}$, then $I_{\fp}$ is of linear type. Hence $K_{\fp}=I_{\fp}$ and the assertion holds trivially. Otherwise, $(K_1)_{\fp} \cap \cdots \cap (K_n)_{\fp} =\core(I_{\fp}) \subset K_{\fp}$. 
\end{proof}

\autoref{coreGLMR} and its proof allows us to show that under the assumptions therein the core of $I$ localizes (cf. \cite[Theorem~4.8]{CPU}).

\begin{corollary}\label{corelocalizes}
If the hypotheses of \autoref{coreGLMR} hold, then $\core(I_{\fp})=(\core(I))_{\fp}\,$ for every $\fp\in {\rm Spec} (R)$. 
\end{corollary}
\begin{proof} 
In the proof of \autoref{coreGLMR} we showed that $\core(I_{\fm_i})=(K_1)_{\fm_i} \cap \cdots \cap (K_n)_{\fm_i}$, which is $\left(\core(I)\right)_{\fm_i}$  by \autoref{coreGLMR}. If $\fp \not\in \{\fm_1, \ldots, \fm_t\}$, then   $I_{\fp}$ is of linear type, therefore   $\core(I_{\fp})=I_{\fp}$ and $(\core(I))_{\fp}=(K_1)_{\fp} \cap \cdots \cap (K_n)_{\fp} =I_{\fp}$.
\end{proof}

\begin{remark}\label{onlym}
If in addition to the assumptions of  \autoref{coreGLMR} the ring $R$ is a positively graded $\kk$-algebra with maximal homogeneous ideal $\fm$ and the ideal  $I$ is homogeneous, then we can replace the assumption that $I$ is weakly $(d-2)$-residually $S_{2}$ by the hypothesis that $I_\fm$ is weakly $(d-2)$-residually $S_{2}$. In addition,  $r$ can be chosen to be $1+r(I_{\fm})$.
\end{remark}
\begin{proof}
Following the proof of \autoref{coreGLMR}, it suffices to show that $I_\fp$ is of linear type whenever $\fp \in V(I)$ and $\fp\neq \fm$.  Since $I$ is homogeneous, the minimal prime ideals of ${\rm Supp}_R(\A)$ are homogeneous and hence are all contained in $\fm$. On the other hand, if $\fp\subsetneq \fm$, then $I_\fp$ is of linear type by our assumption on $I_\fm$ (see \cite[Corollary~3.6  (b)]{CEU}  and \cite[Theorem 2.3.2]{VasBlowup}). 
\end{proof}

In the case of ideals generated by forms of the same degree, we have the following simpler version of \autoref{coreGLMR}.

\begin{corollary}\label{remCo} 
Let $\kk$ be an infinite field,   $R$  a Cohen-Macaulay positively graded $\kk$-algebra of dimension $d$, and  $\fm$  the homogeneous maximal ideal of $R$. Let $I$ be an ideal of positive height generated by homogeneous elements of  the same degree $\delta$.  If $I$ satisfies $G_d$ and $I_\fm$ is weakly $(d-2)$-residually $S_{2}$, then there exists a positive integer $n$ such that    $$\core(I)=J_1\cap\cdots \cap J_n\,,$$
where $J_1,\ldots, J_n$  are generated by $d$ general elements of $I$ with respect to a generating set of $I$ contained in  $I_{\delta}$ $($see \autoref{notation_general}$)$.   
 In particular, $$\core(I)={\rm{gradedcore}}(I).$$
\end{corollary}

\begin{proof}
We chose  $f=0$ as in \autoref{reduc}, so that $J_i=K_i$. Now the assertion  By  \autoref{coreGLMR} and \autoref{onlym}. 
\end{proof}

\begin{remark}\label{rem_analy_d_forced}
The proof of \autoref{coreGLMR} shows that  under the assumptions of \autoref{remCo} either $\ell(I)=d$ or $\core(I)=I$. 
\end{remark}

\begin{question}
Is it possible to replace the assumptions in \autoref{remCo} that $I$ satisfies $G_d$ and $I_\fm$ is weakly $(d-2)$-residually $S_{2}$ by the hypotheses that $I$ satisfies $G_\ell$ and $I_\fm$ is weakly $(\ell-1)$-residually $S_{2}$ for  $\ell=\ell(I)$, to say that   
$$\core(I)=J_1\cap\cdots \cap J_n\,,$$ 
where $J_1,\ldots, J_n$  are generated by $\ell$ general elements of $I$ with respect to a generating set of $I$ contained in  $I_{\delta}$?
\end{question}

\section{The core and the mono}

In this section we generalize the main result of \cite{PUV} to monomial ideals of higher dimension. We show that under suitable  residual conditions, 
the core of a monomial ideal $I$ coincides with the largest monomial ideal in a general reduction of $I$, provided $I$ is of maximal analytic spread. 

\begin{proposition}\label{GB}
Let $\,\kk$ be an infinite field,  $\undl{x}=x_1,\ldots,x_r$, $\undl{y}=y_1,\ldots,y_s$, and $\undl{z}=z_1,\ldots,z_t$ be three sets of variables. Let $H$ be an ideal of $\,\kk[\underline{x},\underline{y},\underline{z}]$. For every $\underline{\lambda}=(\lambda_{i})\in \AA^{t}_\kk$ let $\pi_{\underline{\lambda}}: \kk[\underline{x},\underline{y},\underline{z}] \rightarrow \kk[\underline{x},\underline{y}] $ denote the evaluation map given by $z_i\mapsto \lambda_i$. Then for general $\undl{\lambda}$ we have $\pi_{\underline{\lambda}}(H\cap \kk[\underline{x},\underline{z}])=\pi_{\underline{\lambda}}(H)\cap \kk[\underline{x}]$. 
\end{proposition}

\begin{proof} It is straightforward to see that $\pi_{\underline{\lambda}}(H\cap \kk[\underline{x},\underline{y}])\subseteq \pi_{\underline{\lambda}}(H)\cap \kk[\underline{x}]$.

To prove the reverse inclusion,  we consider the lexicographic monomial  order $<$ on the two polynomial rings $\kk[\undl{z}][\undl{x},\undl{y}]$ and $A:=\kk(\undl{z})[\undl{x}, \undl{y}]$  in the variables $\undl{x}, \undl{y}$ with $x_i<y_j$.  
Let  $G=\{g_1,\ldots, g_m\}\subset  H$ be 
a Gr\"obner basis of $HA$ with respect to $<$.  

Clearly $\pi_{\underline{\lambda}}(G)$ is a generating set of $\pi_{\underline{\lambda}}(H)$ for general $\undl{\lambda}$. We claim that for general $\undl{\lambda}$ the set $\pi_{\underline{\lambda}}(G)$ is also a Gr\"obner basis. 
By Buchberger's criterion it suffices to show that for every $i\neq j$ the $S$-pair $S_{ij}:=S(\pi_{\underline{\lambda}}(g_i), \pi_{\underline{\lambda}}(g_j))$
 is equal to an expression 
 \begin{equation}\label{DivAlg}
 h_1\pi_{\underline{\lambda}}(g_1) + \cdots + h_m\pi_{\underline{\lambda}}(g_m),
 \end{equation}
  where $h_k\in \kk[\undl{x}, \undl{y}]$ and the initial monomials satisfy $\ini{h_k\pi_{\underline{\lambda}}(g_k)} \ls \ini{S_{ij}}$ for every $1\ls k\ls m$. Since $G$ is a Gr\"obner basis of $HA$, there is an expression 
  \begin{equation}\label{DivAlg2}
  S(g_i,g_j)=\widetilde{h}_1g_1 + \cdots + \widetilde{h}_mg_m,
  \end{equation} 
 where $\widetilde{h_k}\in A$ and $\ini{\widetilde{h}_kg_k} \ls \ini{S(g_i,g_j)} $ for every $1\ls k\ls m$.

If $c_{g_k}\in \kk[\undl{z}]$ is the coefficient of $\ini{g_k}$, then $\pi_{\underline{\lambda}}(c_{g_k})$  
is the coefficient of $\ini{\pi_{\underline{\lambda}}(g_k)}$ for general $\undl{\lambda}$.   
Hence 
$\ini{\pi_{\underline{\lambda}}(g_k)}=\ini{g_k}$ and $S_{ij}=\pi_{\underline{\lambda}}(S(g_i, g_j))$ since   $S(g_i, g_j)\in \kk[\undl{x}, \undl{y}, \undl{z}]$. 
Therefore, after clearing denominators in \autoref{DivAlg2} and applying $\pi_{\undl{\lambda}}$ for general $\undl{\lambda}$,  the desired expression for $S_{ij}$ as in \autoref{DivAlg} follows.

Since $\pi_{\underline{\lambda}}(G)$ is a Gr\"obner basis of $\pi_{\underline{\lambda}}(H)$ for general $\undl{\lambda}$ and $<$ is an elimination order, it follows that 
$\pi_{\underline{\lambda}}(H)\cap \kk[\undl{x}]$ is generated by $\pi_{\underline{\lambda}}(G)\cap \kk[\undl{x}]$.  
Finally, for general $\undl{\lambda}$ 
we have $$ \pi_{\underline{\lambda}}(G)\cap \kk[\undl{x}]=\pi_{\underline{\lambda}}\left(G\cap \kk[\underline{x}, \underline{z}]\right)\subseteq \pi_{\underline{\lambda}}\left(H\cap \kk[\underline{x}, \underline{z}]\right),$$ and the conclusion  follows.
\end{proof}

The goal in this section is to show that under suitable assumptions on a monomial ideal $I$, the core of $I$ can be obtained as the mono of a general reduction of $I$, namely $\core(I)=\mono(K)$, where $K$ is a general reduction of $I$ as in \autoref{r general red}  and $\mono(K)$ denotes the largest monomial ideal contained in $K$. In order to compute $\mono(K)$ we follow an algorithm due to Saito, Sturmfels, and Takayama \cite[Algorithm~4.4.2]{SST}.

For the proof of our main result we need a notion of mono of an ideal in a polynomial ring over an arbitrary Noetherian ring.  Let $A$ be a Noetherian ring. 
For an ideal $L$ in the polynomial ring $A[x_1, \ldots, x_d]$,  the {\it{multihomogenization}} of $L$, denoted by $\widetilde{L}$, is  the ideal of the  polynomial ring $A[x_1, \ldots, x_d, y_1, \ldots, y_d]$ generated by
 $$\left\{\widetilde{g}= g\left(\frac{x_1}{y_1},\ldots, \frac{x_d}{y_d}\right)y_1^{\deg_{x_1}(g)}\cdots y_d^{\deg_{x_d}(g)} \Big| \ g\in L\right\}.$$
 We consider  $A[x_1, \ldots, x_d, y_1, \ldots, y_d]$ with the $\NN^d$-grading induced by $\deg(x_i)=\deg(y_i)=\mathbf{e}_i$. 
 We note that $\widetilde{g}$ is indeed multihomogeneous with ${\rm deg}(\widetilde{g})=\left(\deg_{x_1}(g),\ldots, \deg_{x_d}(g)\right)\in \NN^d$.
 
 The next example illustrates the process of multihomogenization of an element.

\begin{example}  Let $g = c_1x_1^2x_2 + c_2x_1x_3^2+c_3x_2^3x_3 \in A[x_1, x_2, x_3]$ for some $c_1,c_2,c_3\in A$.  Then $\widetilde{g}=c_1x_1^2x_2y_2^2y_3^2 + c_2x_1x_3^2y_1y_2^3+c_3x_2^3x_3y_1^2y_3.$
\end{example}

To obtain the multihomogenization of an ideal $L \subset A[x_1, \ldots, x_d]$ it is enough to multihomogenize a given generating set $g_1, \ldots, g_u$ of $L$ and to saturate with respect to $Y=\prod \limits_{j=1}^{d}y_j$, that is 
\begin{equation}\label{Ltilde}
\widetilde{L}=(\widetilde{g_1}, \ldots, \widetilde{g_u}):Y^{\infty}.
\end{equation}

\begin{definition}\label{generalmono}
Let $A$ be a Noetherian ring  and let $L$ be an ideal in the polynomial ring $A[x_1,\ldots, x_d]$. We define  $\mono(L)$ to be the ideal generated by the elements in $L$ of the form $a\, m$, where $a\in A$   and $m$ is a monomial.
\end{definition}

Following  \cite[Algorithm 4.4.2]{SST}, we obtain 
\begin{equation}\label{monoeqn}
\mono(L)=\widetilde{L}\, \cap \, A[x_1,\ldots, x_d].
\end{equation}

\begin{proposition}\label{indep}
Let $A=\kk[z_1, \ldots, z_t]$ be a polynomial ring over an infinite field $\kk$ and $L$ a proper ideal in the polynomial ring  $A[x_1,\ldots, x_d]$.  For $\underline{\lambda}\in \AA^{t}_\kk$  let
$\pi_{\underline{\lambda}}:A[x_1,\ldots, x_d]\rightarrow  \kk[x_1,\ldots, x_d]$
be  the evaluation map given by $z_i\mapsto \lambda_i$. For general $\undl{\ll}\in \AA^{t}_\kk$ we have the following 
\begin{enumerate}[$($a$)$]
\item   $\widetilde{\pi_{\undl{\ll}}(L)}=\pi_{\undl{\ll}}(\widetilde{L})$
\item   $\mono(\pi_{\undl{\ll}}(L))=\pi_{\undl{\ll}}(\mono(L))$
\item  $\mono(\pi_{\undl{\ll}}(L))$ does not depend on $\undl{\ll}\,$.
\end{enumerate}
\end{proposition}

\begin{proof} 
We begin with the proof of (a). We first notice that for any $g\in A[x_1,\ldots, x_d]$ and general $\undl{\ll}$ we have $\widetilde{\pi_{\undl{\ll}}(g)}=\pi_{\undl{\ll}}(\widetilde{g})$. Write $L=(g_1,\ldots,g_u)$, then 
$$\left(\widetilde{\pi_{\undl{\ll}}(g_1)},\ldots,\widetilde{\pi_{\undl{\ll}}(g_u)}\right)=
\left(\pi_{\undl{\ll}}(\widetilde{g_1}),\ldots, \pi_{\undl{\ll}}(\widetilde{g_u})\right)
=
\pi_{\undl{\ll}}(\widetilde{g_1},\ldots, \widetilde{g_u}).
$$
Therefore, $$\widetilde{\pi_{\undl{\ll}}(L)}=\left(\widetilde{\pi_{\undl{\ll}}(g_1)},\ldots,\widetilde{\pi_{\undl{\ll}}(g_u)}\right):Y^\infty
= \pi_{\undl{\ll}}(\widetilde{g_1},\ldots, \widetilde{g_u}):Y^\infty.$$
On the other hand,
$$\pi_{\undl{\ll}}(\widetilde{L})=
 \pi_{\undl{\ll}}((\widetilde{g_1},\ldots, \widetilde{g_u}):Y^\infty).$$
Notice that
 $$\pi_{\undl{\ll}}(\widetilde{g_1},\ldots, \widetilde{g_u})\subseteq 
 \pi_{\undl{\ll}}((\widetilde{g_1},\ldots, \widetilde{g_u}):Y^\infty)
 \subseteq
 \pi_{\undl{\ll}}(\widetilde{g_1},\ldots, \widetilde{g_u}):Y^\infty.
 $$
 Thus to prove that $\widetilde{\pi_{\undl{\ll}}(L)}=\pi_{\undl{\ll}}(\widetilde{L})$ it suffices to show that  
$ \pi_{\undl{\ll}}((\widetilde{g_1},\ldots, \widetilde{g_u}):Y^\infty)$ is saturated with respect to $Y$, for this it suffices to show that $Y$ is a non-zerodivisor on $\kk[\undl{x},\undl{y}]/\pi_{\undl{\ll}}((\widetilde{g_1},\ldots, \widetilde{g_u}):Y^\infty)=\kk[\undl{x},\undl{y}]/\pi_{\undl{\ll}}(\widetilde{L}).$ 
The image of $Y$  is not a unit in $\kk[\undl{x},\undl{y}]/\pi_{\undl{\ll}}(\widetilde{L})$, hence 
$
 \kk[\undl{x},\undl{y}]/(\pi_{\undl{\ll}}(\widetilde{L}),Y) \neq 0.$

Set $$T= \frac{ \kk[\undl{z}]_{(\undl{z}-\undl{\ll})}[\undl{x},\undl{y}]  }{ \widetilde{L}}$$ 
 and notice that $T/(\undl{z}-\undl{\ll})=\kk[\undl{x},\undl{y}]/\pi_{\undl{\ll}}(\widetilde{L}).$
  By generic freeness   \cite[Theorem 14.4]{E}, for general $\undl{\ll}$ the map $\kk[\undl{z}]_{(\undl{z}-\undl{\ll})}\rightarrow T/(Y)$ is flat and hence the elements $\undl{z}-\undl{\ll}$ form a regular sequence on $T/(Y)$. For this also recall that 
 $$T/(Y, \undl{z}-\undl{\ll})=
 \kk[\undl{x},\undl{y}]/(Y,\pi_{\undl{\ll}}(\widetilde{L})) \neq 0\,.$$
 Since $Y$ is a non-zerodivisor on $T$ it follows that $Y, \undl{z}-\undl{\ll}$ is a $T$-regular sequence. As this sequence consists of homogeneous elements in $T$, and $T$ is a positively graded ring over a local ring, we obtain that $\undl{z}-\undl{\ll}, Y$ is also a regular sequence \cite[Theorem 16.2 and Theorem 16.3]{Mat}. We conclude that $Y$ is a non-zerodivisor on $T/(\undl{z}-\undl{\ll})=\kk[\undl{x},\undl{y}]/\pi_{\undl{\ll}}(\widetilde{L}).$

Part $(b)$ is a direct consequence of $(a)$ and \autoref{GB}. 

Finally, part $(c)$ follows from $(b)$ because $ \pi_{\undl{\ll}}(\mono(L))$ does not depend on $\undl{\ll}$ for general $\undl{\ll}$. Indeed, if $\{a_i\, m_i\}$ is a finite generating set of $\mono(L)$, where $a_i\in \kk[\undl{z}]$ and $m_i$ are monomials in $x_1,\ldots, x_d$, then for any $\undl{\ll}\in D(\prod_i a_i)$ the ideal $\pi_{\undl{\ll}}\big(\mono(L)\big)$ is independent of $\undl{\ll}$. 
\end{proof}

\begin{corollary}\label{maincor}
Let $\kk$ be an infinite field, $R=\kk[x_1,\ldots, x_d]$ a polynomial ring, and  $I$ a monomial ideal.   For any $n\in \NN$ let  $K$ and $K_1, \ldots, K_n$ be general reductions of $I$   as in \autoref{remark_r}.  We have
 $$\core(I) \subseteq \mono(K) \subseteq K_1 \cap \cdots \cap K_n\,.$$
\end{corollary}

\begin{proof}
Clearly,  $\core(I)\subseteq \mono(K)$, since $K$ is a reduction of $I$ by \autoref{reduc} and  $\core(I)$ is a monomial ideal by \cite[proof of Remark 5.1]{CPU}. For the second inclusion, notice that $\mono(K)=\mono(K_i)$ for all $1\ls i \ls n$  according to \autoref{indep}(c).
\end{proof}

\begin{remark}\label{GC}
If in  \autoref{maincor}, the ideal $I$ is generated by monomials of degree $\delta$ and the elements $f_1,\ldots, f_u$ of  \autoref{Notation} are chosen to be homogeneous polynomials of degree $\delta$ and $f=0$,  then $K$ is a homogeneous reduction of $I$. Therefore $$\gradedcore(I)\subseteq \mono(K).$$
\end{remark}

We now prove  the main theorem of this section.

\begin{theorem}\label{MainwithGdS2}
Let $\kk$ be an infinite field,  $R=\kk[x_1,\ldots, x_d]$ a polynomial ring,  $\fm:=(x_1,\ldots, x_d)$, and  $I$ a monomial ideal. 
 If  $I$ satisfies $G_d$ and $I_\fm$ is weakly $(d-2)$-residually $S_{2}$, then
 $$\core(I)=\mono(K)$$
 for $K$ a general reduction of $\,I$  with $r=d$ as in \autoref{remark_r}.
\end{theorem}
\begin{proof}
The proof follows by \autoref{maincor}, \autoref{coreGLMR},  and \autoref{onlym}. 
\end{proof}

The following  example shows that \autoref{coreGLMR} and \autoref{MainwithGdS2} do not hold without the assumption that $I$ is $G_d$.

\begin{example}\label{ex1}
Let $R=\mathbb{Q}[x_1,x_2,x_3]$ and $I=(x_1^3,x_1^2x_2,x_1x_3^2,x_3^3)$.  The ideal $I$ has height 2 and analytic spread 3. It is  weakly 2-residually $S_2$ because every link of $I$ is unmixed and hence Cohen-Macaulay.  However, $I$ does not satisfy $G_3$. 
Computation with Macaulay2 \cite{GS} shows that there exist non-zero polynomials $h$ and $g$ in  $\mathbb{Q}[\undl{z}]$ such that 
$$\mono(J_{3,\,\undl{z}}) = (h)\left(x_1^2,x_1x_2,x_1x_3,x_2x_3,x_3^2\right)I+ (hg)\left(x_1^2x_2^3, x_1x_2^2x_3^2, x_2^2x_3^3\right)=: (h)\mathfrak{A}+(hg)\mathfrak{B}.$$
For general $\undl{\lambda}$ we have $\mathfrak{A}+\mathfrak{B}=\pi_{\undl{\lambda}}(\mono(J_{3,\,\undl{z}}))=\mono(K)$, where $K:=\pi_{\undl{\lambda}}(J_{3,\,\undl{z}})$ (see \autoref{indep}). 
Therefore $ \core(I)\subseteq \gradedcore(I) \subseteq \mathfrak{A}+\mathfrak{B}= I\,\fm^2$ (see \autoref{GC}).  
The ideal $H=(x_1^3, x_1^2x_2, x_1x_3^2+x_3^3)$ 
 is a minimal reduction of $I$, since $I^3=HI^2$. On the other hand, $\mono(K)=I\,\fm^2 \not\subseteq H$. Hence,  $\gradedcore(I)$ is not equal to $\mono(K)$ for a general reduction $K$ and thus $\core(I)$ is not a finite intersection of general reductions of $I$ (see \autoref{maincor}). 
 In particular, neither \autoref{coreGLMR}  nor \autoref{MainwithGdS2} hold. 
\end{example}

The ideal in the next example  is $G_d$ (in fact $G_\infty$), but $\ell(I)<d$ and  $I_\fm$ is not  weakly $(d-2)$-residually $S_{2}$ (see \autoref{rem_analy_d_forced} and \autoref{remCo}). Again,  
$\core(I)$ is not a finite intersection of general minimal reductions of $I$ and it is not the mono of a general minimal reduction of $I$.

\begin{example}\label{ex2}
Let $$R = \QQ[x_1,x_2,x_3,x_4,x_5,x_6]\quad\text{ and }\quad I= (x_1x_2,x_2x_3,x_3x_4,x_4x_1,x_4x_5,x_5x_6).$$  
One easily verifies  that the height of $I$ is $3$ and that it satisfies $G_\infty$. 
However, $\ell(I)=5$ and $I_\fm$ is not weakly 3-residually $S_2$. 
In \cite[Example 4.8]{FM} it is shown that $\core(I) \neq \fm\, I$. Using Macaulay2~\cite{GS} one verifies that $\mono(J)=\fm\, I$, for $J$ a general minimal reduction of $I$, i.e., an ideal generated by 5 general elements of $I$ with respect to the six monomial generators of $I$. 
Therefore  
$\core(I)$   is not equal to $\mono(J)$. 
\end{example}

\smallskip

\section{ The core of monomial ideals generated in one degree}\label{mono section}

There is no known method to compute the core  of a given  ideal if  the residual conditions required in the previous sections do not hold.  
Our goal in this section is to propose an approach to compute the core of monomial ideals generated in a single degree without any further assumptions.

In the previous section we established that for a monomial ideal $I$,   
$\core(I)$ and  $\gradedcore(I)$ are contained in $\mono(K)$ for a general reduction $K$ of  $I$. This containment  holds in general and it is an equality under appropriate residual conditions (see \autoref{maincor}, \autoref{GC}, \autoref{MainwithGdS2}).  
 For a monomial ideal $I$ generated in a single degree we construct an ideal that is contained in $\gradedcore(I)$ 
 and we conjecture that equality holds in general (see \autoref{mainOneDegree} and \autoref{conjectureColon}). 
 We verify the  conjecture for a specific ideal  in \autoref{ex3}. 
Furthermore, under the same residual conditions $\core(I)$ can be obtained as the intersection of finitely many general reductions \autoref{remCo}. However, as seen in \autoref{ex1}, {\it special} reductions are needed in the absence of the residual conditions. In this section we provide a method to find these special reductions.

\begin{notation}\label{Notation3}
Let $\kk$ be an infinite field,  $R=\kk[x_1,\ldots, x_d]$  a polynomial ring, and $\fm=(x_1,\ldots, x_d)$  its homogeneous maximal ideal. Let $I$ be a non-zero ideal generated by homogeneous elements  of the same degree $\delta$.  Let  $\F := \F(I)$ 
be the  special fiber ring of $I$ and $\F_+$ the ideal generated by the elements of  $\F$ of positive degree.  Notice that $\F\cong \kk[I_\delta]\subseteq R$ because $I$ is generated in a single degree. Let   $\ell:=\ell(I)=\dim(\F)$ be the  analytic spread of $I$. 

Fix a generating sequence $\undl{f}=f_1,\ldots, f_u$ of $I$ contained in $I_\delta$. 
 Consider $\ell u$ variables $\underline{z}=\{z_{ij}\mid 1\ls i\ls \ell \text{ and  }1\ls j\ls u\}$.  
 Write $b_i = \sum_{j=1}^u z_{i,j}f_j $. Let $\H\subseteq J:=J_{\ell,\undl{z}}(\undl{f})$ 
 be the ideals generated by $b_1,\ldots, b_\ell$ in the rings $\F[\undl{z}]\subseteq R[\undl{z}]$, respectively. For $\undl{\ll}\in \AA_\kk^{\ell u}$, we write $\H_{\undl{\ll}}=\pi_{\undl{\ll}}(\H)\subseteq \F$ and $J_{\undl{\ll}}=\pi_{\undl{\ll}}(J)\subseteq R$, where $\pi_{\undl{\ll}}$ denotes the evaluation map.

Notice that $\F_+ R =I$, $\H R[\undl{z}]=J$, and $\H_{\undl{\ll}}R=J_{\undl{\ll}}$.  
Moreover, $J_{\undl{\ll}}$ is a reduction of $I$ if and only if $I^{r+1}=J_{\undl{\ll}}I^r$ for some $r\gs 0$  if and only if $\F_{+}^{r+1}=\H_{\undl{\ll}}\F_{+}^r$ for some $r\gs 0$ if and only if $\F_{+} \subseteq \sqrt{\H_{\undl{\ll}}}$. 
\end{notation}

The following result describes the locus of the points $\undl{\lambda}$ for which  $J_{\undl{\lambda}}$ is not a reduction of $I$.  We  also show that this locus is determined by a single irreducible polynomial of $\kk[\undl{z}]$. We note that here we only assume $I$ is  homogeneous and not necessarily generated by monomials.

\begin{proposition}\label{nonRedLocus} With assumptions as in \autoref{Notation3} let $\A=(\H:_{\F[\undl{z}]} \F_+^\infty)\cap \kk[\undl{z}]$. 
\begin{enumerate} [$($a$)$]
\item The  $\kk[\undl{z}]$-ideal $\A$ defines the locus where $J_{\undl{\lambda}}$ is not a reduction of $I$.
\item The ideal $\A$ is a prime ideal of height one. Thus $\A=(h)$, where $h$ an irreducible polynomial in $\kk[\undl{z}]$.
\end{enumerate}
\end{proposition}
\begin{proof} 
To prove (a) we write  $T=\F[\undl{z}]/\H$ and consider the natural map $\phi : \Proj (T) \rightarrow {\rm{Spec}}(\kk[\undl{z}])$. 
Clearly  ${\rm Im}(\phi)\subseteq V(\A)$; we claim that $V(\A)={\rm Im}(\phi)$. For this, we first  note that  $\,\H:_{\F[\undl{z}]} \F_+^\infty\,$ is a  prime ideal of height $\ell$ by  \autoref{prime}. Therefore, we have an inclusion of domains 
\begin{equation*}\label{incl_domains}
U:=\kk[\undl{z}]/\A\xhookrightarrow{\,\,\,\,\, \,\,\,\,\, } V:=\F[\undl{z}]/(\H:_{\F[\undl{z}]} \F_+^\infty).
\end{equation*}
Since 
\begin{equation}\label{bigger_than_one}
	\dim\left( V\otimes_{U}{\rm Quot}(U) \right)=\height(\F_+V)\gs 1, 
\end{equation}	
by semicontinuity of fiber dimension \cite[Theorem 14.8 (b)]{E} we have that $\dim\left( V\otimes_U\kappa(P)\right)\gs 1$ for every $P\in \Spec(U)$. Therefore $P\in {\rm Im}(\phi)$ for every $P\in \Spec(U)$, whence the claim follows. 

A point $\undl{\ll}\in \AA_\kk^{\ell u}$ belongs to ${\rm Im}(\phi)$ if and only if 
 $\dim \left(T\otimes_{\kk[\undl{z}]}  \left(\kk[\undl{z}]/ {(\zb-\llb)}\right)\right) >0$. 
 Since $$T\otimes_{\kk[\undl{z}]}  \left(\kk[\undl{z}]/ {(\zb-\llb)}\right)\cong \F/\H_{\llb},$$ the last condition is equivalent to $\F_{+} \not\subseteq \sqrt{\H_{\llb}}$, which means that $J_{\llb}$ is not a reduction of $I$.

For part (b), it remains to show that $\height(\A)=1$. 
 We first observe that 
\begin{eqnarray*}
\dim(V)  = \dim\left(\F[\undl{z}]\right) - \height\left(\H:_{\F[\undl{z}]}\F_+^\infty\right) 
=\left(\ell+ \dim(\kk[\undl{z}])\right) - \ell=\dim(\kk[\undl{z}]).
\end{eqnarray*} 
We think of points in $\AA_\kk^{\ell u}$ as $\ell \times u$ matrices.
 If $\llb_0 \in \AA_\kk^{\ell u}$ is a matrix whose first $\ell-1$  rows are general and whose last row consists of zeros, then $\height\left(\H_{\llb_0}\right)=\ell -1$. 
  Therefore,  $\F/\H_{{\llb}_0}=V\otimes_U \left(U/(\zb-{\llb_0})\right)$ has dimension one,  
  which shows that $\dim\left(V\otimes_{U}{\rm Quot}(U)\right)= 1$ by  \cite[Theorem 14.8 (b)]{E} and \autoref{bigger_than_one}. Thus, $$1= \trdeg_U(V) = \trdeg_\kk(V) - \trdeg_\kk(U)= \dim(V)  - \dim(U)=\dim(\kk[\undl{z}])-\dim(U)= \height(\A),$$
  completing the proof. 
\end{proof}

Following \autoref{Notation3}, one can see that every homogeneous reduction of $I$ contains a reduction generated by $\ell$ homogeneous elements of degree $\delta$, which is necessarily of the form $J_{\llb}$ for some $\llb \in \AA_{\kk}^{\ell u}$. 

\begin{corollary}
With assumptions as in \autoref{Notation3} and \autoref{nonRedLocus}, we have
$$\gradedcore (I) =\bigcap \limits_{\llb \not\in V(\A)} J_{\llb}.$$  
\end{corollary}

In the  following  result  we show that for a monomial ideal $I$,     $\left(\mono(J):_{R[\undl{z}]}(h)^\infty\right)\cap R$ is contained in every homogeneous reduction  of $I$.  
In fact,  we conjecture that this ideal is equal to $\gradedcore(I)$ (see \autoref{conjectureColon}). 
Here we think of $J$ as an ideal in the polynomial ring $A[x_1,\ldots, x_d]$ with $A=\kk[\undl{z}]$ (see \autoref{generalmono}).  For a vector $\mathbf{w}=(w_1,\ldots, w_d)\in \NN^d$ we denote by $\fx^\mathbf{w}$ the monomial $x_1^{w_1}\cdots x_d^{w_d}$.

\begin{theorem}\label{mainOneDegree}
	In addition to the assumptions of   \autoref{Notation3}    
    we suppose that 
	$f_1, \ldots, f_u$ are monomials. 
	If $h\in\kk[\undl{z}]$ is as in \autoref{nonRedLocus}$\, ($b$)$, then $$\left(\mono(J):_{R[\undl{z}]}(h)^\infty\right) \cap R \subseteq \gradedcore(I).$$
\end{theorem}
\begin{proof}
	Let $\fx^\fv \in \pars{\mono(J):_{R[\undl{z}]}(h)^\infty} \cap R$. Then  $\fx^\fv h^N\in \mono(J)$ for $N\gg 0$. By \autoref{nonRedLocus}  for each $\undl{\lambda}$ such that  $J_{\undl{\lambda}}$  is a reduction of $I$ we have $\pi_{\undl{\lambda}}(h)\neq 0$. Hence, setting $Y=\prod y_i$ as in \autoref{Ltilde} and using \autoref{monoeqn}, we obtain
	\begin{align*}
		\fx^\fv\in \pi_{\undl{\lambda}}(\mono(J)) &= \pi_{\undl{\lambda}}\left( \left((\widetilde{b_1}, \ldots, \widetilde{b_{\ell}}):_{R[\undl{z},\undl{y}]}Y^{\infty}\right)\cap R[\undl{z}]\right)\\
		&\subseteq \pi_{\undl{\lambda}}\left((\widetilde{b_1}, \ldots, \widetilde{b_{\ell}}):_{R[\undl{z},\undl{y}]}Y^{\infty}\right)\cap R\\
		&\subseteq (\pi_{\undl{\lambda}}(\widetilde{b_1}, \ldots, \widetilde{b_{\ell}}):_{R[\undl{y}]}Y^{\infty})\cap R=\mono(J_{\undl{\lambda}})\subseteq J_{\undl{\lambda}}.
	\end{align*}
	Taking the intersection over all such $\undl{\lambda}$ we obtain $\fx^\fv\in \gradedcore(I)$, as desired.
\end{proof}

We propose the following conjecture based on the previous result and computational evidence.

\begin{conjecture}\label{conjectureColon}
	Let $I$ and $h$ be as in \autoref{mainOneDegree}. Then $$\gradedcore(I)=(\mono(J):_{R[\undl{z}]}(h)^\infty) \cap R.$$
\end{conjecture}

In our next result, we show that the content ideal of $\mono(J)$ is principal and that it is generated by the irreducible polynomial $h$ from \autoref{nonRedLocus} (b).  We use this result to verify \autoref{conjectureColon} for  specific examples at the end of the section.  
Before we proceed we need to fix more notation. 

\begin{notation}\label{Notation4} 	In addition to the assumptions of   \autoref{Notation3}    
    we suppose that 
	$f_1, \ldots, f_u$ are monomials. Let  $\fv_1,\ldots, \fv_r\in \NN^d$ be distinct vectors such that  $\mono(J) = C_1(\fx^{\fv_1}) + \ldots+C_r(\fx^{\fv_r})$, where  $C_1,\, \ldots, C_r $ are ideals of $\kk[\undl{z}]$ (see \autoref{generalmono}). The ideal $\C=C_1+\ldots+C_r$ is called the {\it content ideal} of $\mono(J)$. 
     We note that the set of monomials $\mathcal{M}:=\{\fx^\fv_1,\ldots, \fx^\fv_r\}$ generates $\mono(J_{\undl{\lambda}})$ 
     for general $\undl{\lambda}$  by \autoref{indep} (b).

Let $\fx^\fv=\lcm(f_1,\ldots, f_u)$ and for each $f_i$ let $g_i$ be the monomial in $R[\undl{y}]=R[y_1,\ldots, y_d]$ such that  $\multdeg(g_i)=\fv$ and $\multdeg_{x_j}(g_i)=\multdeg_{x_j}(f_i)$ for every $i$ and $j$.  Notice that $\sum_{j=1}^u z_{i,j}g_j$ is $\widetilde{b_i}$, the multihomogenization of $b_i \in \kk[\undl{z}][x_1, \ldots, x_d]$. 
Let $\widehat{f_i}$ be the element of $\kk[\undl{y}]$ such that $\multdeg_{y_j}(\widehat{f_i})=\multdeg_{y_j}(g_i)$, that is $g_i=f_i\widehat{f_i}$, and set $\widehat{I}=(\widehat{f_1},\ldots, \widehat{f_u}) \subseteq \kk[\undl{y}]$. The ideal $\widehat{I}$ is the {\it Newton complementary dual} of $I$ defined in \cite{CS} (see also \cite{ALS}).
\end{notation}

\begin{theorem}\label{ContentIsPrincipal}
Let $\C$ be as in \autoref{Notation4} and let $h\in \kk[\undl{z}]$ be  as in \autoref{nonRedLocus} $($b$)$. Then $\C=(h).$
\end{theorem}
\begin{proof}
We prove the result by constructing a $\kk[\undl{z}]$-isomorphism 
$$
\begin{tikzcd}
\eta: \dfrac{\kk[\undl{z}]}{(h)} \arrow[r, rightarrow, "\sim"] & 
\dfrac{\kk[\undl{z}]}{\C}\, .
\end{tikzcd}
$$
For this we consider the following diagram, which we explain in the rest of the proof. 

\begin{center}
\begin{tikzcd}
&\hspace{-3.5cm}\dfrac{T}{(\widetilde{b_1}, \ldots, \widetilde{b_{\ell}}):Y^\infty}\stackrel{(1)}{=}\dfrac{T}{(\widetilde{b_1}, \ldots, \widetilde{b_{\ell}}):G^\infty}&\\

\hspace{-0.5cm}\dfrac{\F[\undl{z}]}{\H:\F_+^\infty} \arrow[r,  "\overline{\varphi}"] & \dfrac{S}{(\widetilde{b_1}, \ldots, \widetilde{b_{\ell}}):G^\infty} \arrow[u, hook,  "\psi"] \arrow{r} {\overline{\phi}}[swap]{\sim}&\dfrac{\F(\widehat{I})[\undl{z}]}{(\widehat{b_1}, \ldots, \widehat{b_\ell}): \widehat{F}^{\infty}}\,[\undl{x}, \undl{x}^{-1}] \\

\hspace{-0.5cm} \arrow[u, hook] \dfrac{A}{(h)} \arrow[dr, two heads, "\eta" '] 
& \hspace{-0.5cm}\dfrac{A[\undl{x}, \undl{x}^{-1}]}{ \mono(J)\,A[\undl{x}, \undl{x}^{-1}]}\arrow[u, hook, "\chi"]= \dfrac{A}{\C }\,[\undl{x}, \undl{x}^{-1}] \arrow{r} {\Phi}[swap]{\sim}&\dfrac{A}{(q)}\,[\undl{x}, \undl{x}^{-1}] \arrow[u, hook]\\ 

& \dfrac{A}{\C}\arrow[u, hook]& \dfrac{A}{(q)}\arrow[u, hook]
\end{tikzcd}
\end{center}

Write $$A:=\kk[\undl{z}] \subseteq S:=A[\undl{x}, \undl{x}^{-1}][g_1, \ldots, g_u] \subseteq T:=A[\undl{x}, \undl{x}^{-1}][y_1, \ldots, y_d],$$ 
and set $\,Y=\prod \limits_{j=1}^{d}y_j, \ F=\prod \limits_{j=1}^{u}f_j, \ G=\prod \limits_{j=1}^{u}g_j$, and $\widehat{F}=\prod \limits_{j=1}^{u}\widehat{f_j}$.

The equality (1) at the top of the diagram follows from \autoref{prime}  since $Y\in \sqrt{(g_1, \ldots, g_u)\kk[\undl{x}, \undl{x}^{-1}][y_1, \ldots, y_d]}$ as $x_i$ are units.

We continue by constructing the map $\psi$. The inclusion $S\subset T$ induces an $A$-algebra homomorphism 
$$	\begin{tikzcd}
\psi: \dfrac{S}{(\widetilde{b_1}, \ldots, \widetilde{b_{\ell}}):_S G^\infty}  
\arrow[r, rightarrow] &
\dfrac{T}{(\widetilde{b_1}, \ldots, \widetilde{b_{\ell}}):_TG^\infty}\, . 
\end{tikzcd}
$$ 
We claim  $\psi$ is injective. Since $G$ is a non-zerodivisor modulo  
$(\widetilde{b_1}, \ldots, \widetilde{b_{\ell}}):_SG^\infty$, it suffices to show that $\psi \otimes_S S_G$ is injective. 
Write $z_{i}'=z_{i,1}+g_1^{-1}\sum_{j=2}^uz_{i, j }g_j\in S_G$ for $1\ls i\ls \ell$. 
Notice that 
$$\left((\widetilde{b_1}, \ldots, \widetilde{b_{\ell}}):_SG^\infty\right)_G=(\widetilde{b_1}, \ldots, \widetilde{b_{\ell}}) S_G= (z_1', \ldots, z_{\ell}')S_G\, $$
and similarly 
$$\left((\widetilde{b_1}, \ldots, \widetilde{b_{\ell}}):_TG^\infty\right)_G=  (z_1', \ldots, z_{\ell}')T_G\, .$$
Consider the two rings
 $$B:=\kk[\undl{x},  \undl{x}^{-1}][\undl{g},  G^{-1}][\{z_{i,j} \mid j\gs 2\}] \subseteq C:=\kk[\undl{x},  \undl{x}^{-1}][\undl{y},  G^{-1}][\{z_{i,j} \mid j\gs 2\}].$$
One has $S_G=B[z_1', \ldots, z_{\ell}']$ and $T_G=C[z_1', \ldots, z_{\ell}']$, and $z_1', \ldots, z_{\ell}'$ are variables over $B$ and $C$. Clearly,	
$$	\begin{tikzcd}
			\dfrac{B[z_1', \ldots, z_{\ell}']}{(z_1', \ldots, z_{\ell}')}   \arrow[r, hookrightarrow] & \dfrac{C[z_1', \ldots, z_{\ell}']}{(z_1', \ldots, z_{\ell}')} \, ,
	\end{tikzcd}$$
which proves the claim. 

Next we deal with the map $\overline{\varphi}$. 
Define a map of $A$-algebras 
$$\varphi: \F[\undl{z}]\rightarrow S$$
given by $\varphi(f_i)=g_i$. 
To prove that $\varphi$ is well-defined, let $p$ be a polynomial with coefficients in $A$ such that $p(f_1,\ldots, f_u)=0$ and fix $\fw=(w_1,\ldots, w_d)\in \NN^d$. 
Let $p_\fw$ be the sum of the terms $p'$ of $p$ such that $\multdeg\left(p'(f_1,\ldots, f_u)\right)=\fw$. Therefore $$p_\fw(g_1,\ldots, g_u)=\fy^{\frac{\sum_i w_i}{\delta}\fv-\fw}p_\fw(f_1,\ldots, f_u)=0.$$ We conclude that $p(g_1,\ldots, g_u)=0$ showing that $\varphi$ is well-defined. 
Notice that $\varphi(b_i)=\sum \limits_{j=1}^{u} z_{i,j}g_j=\widetilde{b_i}$, hence $\varphi(\H)\subseteq (\widetilde{b_1}, \ldots, \widetilde{b_\ell})$. 
 Therefore, we have
$\varphi(\H:_{\F[\undl{z}]} F^{\infty})\subseteq (\widetilde{b_1}, \ldots, \widetilde{b_{\ell}}) :_{S} G^{\infty}$. 
Now \autoref{prime} shows that $\H:_{\F[\undl{z}]}\F_+^\infty =\H:_{\F[\undl{z}]} F^{\infty}$. 
It follows that $\varphi$ induces a homomorphism of $A$-algebras 
$$
\begin{tikzcd}
\overline{\varphi}: \  \dfrac{\F[\undl{z}]}{\H:\F_+^\infty}
\arrow[r, rightarrow] &  
\dfrac{S}{(\widetilde{b_1}, \ldots, \widetilde{b_{\ell}}):G^\infty}\,.
\end{tikzcd}
$$

Now we construct the isomorphism $\overline{\phi}$.
Notice that $\widehat{f_j}=a_j g_j$ where $a_j=f_j^{-1} \in \kk[\undl{x}, \undl{x}^{-1}]$ is a unit, in particular $\widehat{F}$ is equal to $G$ times a unit in $\kk[\undl{x}, \undl{x}^{-1}]$. Now 
 $$S=\kk[\widehat{f_1}, \ldots, \widehat{f_u}][\undl{z}][\undl{x}, \undl{x}^{-1}]=\F\big(\widehat{I}\,\,\big)[\undl{z}][\undl{x}, \undl{x}^{-1}]\,.$$

 Consider the automorphism $\phi$ of $S$ as an algebra over  $\kk[g_1, \ldots, g_u][\undl{x}, \undl{x}^{-1}]$ 
 that sends $z_{i,j}$ to $a_jz_{i,j}$. 
 Notice that 
$\phi$ maps $A[\undl{x}, \undl{x}^{-1}]$ onto itself and
sends $\widetilde{b_i}$ to $\widehat{b_i}:=\sum \limits_{j=1}^{u} z_{i,j}\widehat{f_j}$. Hence $\phi$ induces an isomorphism 
$$\overline{\phi}: \  \dfrac{S}{(\widetilde{b_1}, \ldots, \widetilde{b_{\ell}}):G^\infty}\  \longrightarrow \ \dfrac{\F\big(\widehat{I}\,\,\big)[\undl{z}]}{(\widehat{b_1}, \ldots, \widehat{b_\ell}):\widehat{F}^\infty}\,[\undl{x}, \undl{x}^{-1}] $$
that maps the image of $A[\undl{x}, \undl{x}^{-1}]$ onto itself.

We now deal with the map $\chi$. Recall that 
$
(\widetilde{b_1}, \ldots, \widetilde{b_{\ell}}):_{T}Y^\infty =
(\widetilde{b_1}, \ldots, \widetilde{b_{\ell}}):_{T}G^\infty
$
by the equality (1) at the top of the diagram. 
Hence the inclusion $A[\undl{x}, \undl{x}^{-1}] \subseteq S \subseteq T$ induces the natural embedding 
$$	\begin{tikzcd}
\chi: \dfrac{A[\undl{x}, \undl{x}^{-1}]}{\left( (\widetilde{b_1}, \ldots, \widetilde{b_{\ell}}):_{T}Y^\infty\right)\cap A[\undl{x},\undl{x}^{-1}]}  \arrow[r, hookrightarrow] &   \dfrac{S}{(\widetilde{b_1}, \ldots, \widetilde{b_{\ell}}):G^\infty}\, .	\end{tikzcd}$$
On the other hand, 
\begin{eqnarray*}
	\left( (\widetilde{b_1}, \ldots, \widetilde{b_{\ell}}):_{T}Y^\infty\right)\cap A[\undl{x}, \undl{x}^{-1}]&=&\left(\left( (\widetilde{b_1}, \ldots, \widetilde{b_{\ell}}):_{A[\undl{x}, \undl{y}]}Y^\infty\right)\cap A[\undl{x}])\right)A[\undl{x}, \undl{x}^{-1}]\\
&=&\mono (J)\, A[\undl{x}, \undl{x}^{-1}]=\C \, A[\undl{x}, \undl{x}^{-1}]\, ,
\end{eqnarray*}
where the penultimate equality holds by \eqref{monoeqn}. 

We continue by establishing the isomorphism $\Phi$. 
Since the isomorphism $\overline{\phi}$ maps the image $A[\undl{x}, \undl{x}^{-1}]$ onto itself, it follows that this map restricts to an isomorphism
$$
\begin{tikzcd}
\Phi: \ \dfrac{A}{\C }\,[\undl{x}, \undl{x}^{-1}] \arrow[r, rightarrow] &  
\dfrac{A}{\left( (\widehat{b_1}, \ldots, \widehat{b_\ell}):\widehat{F}^{\infty}\right) \cap A}\,[\undl{x}, \undl{x}^{-1}]\, . 
\end{tikzcd}
$$
By \autoref{nonRedLocus} (b) and \autoref{prime}, the ideal $\left((\widehat{b_1}, \ldots, \widehat{b_\ell}):_{\F(\widehat{I})[\undl{z}]}\widehat{F}^{\infty}\right)\cap A$ is generated by an irreducible polynomial $q$.

Finally we construct the desired map $\eta$. 
Recall that 
$(\H: \F_+^\infty)\cap A=(h)$
 by \autoref{nonRedLocus} (b).  
Since $\overline{\varphi}$  is a homomorphism of $A$-algebras, it induces an epimorphism of $A$-algebras 
$$
\begin{tikzcd}
\eta: \dfrac{A}{(h)} \arrow[r, twoheadrightarrow] & 
\dfrac{A}{\C}\, .
\end{tikzcd}
$$
It follows that $(h) \subseteq \C$.  On the other hand,  $$\height (\C)=\height\left( \C A[\undl{x}, \undl{x}^{-1}]\right)=\height\left( \, q \,A[\undl{x}, \undl{x}^{-1}]\right)=\height(q) = 1,$$
where the second equality holds because of the isomorphism $\Phi$. 
Since $(h)$ is a prime ideal of height 1, we conclude that $\C=(h)$, finishing the proof. 
\end{proof}

 The variation of the coefficient ideals occurring
in ${\text {mono}}(J)$ provides a tool to
distinguish between the monomials of $\mathcal M$ and possibly single
out the relevant ones: 

\begin{discussion}\label{discus_one_degree}
With assumptions as in \autoref{Notation4} and \autoref{ContentIsPrincipal}. 
We are now in a position to single out the set
$\mathcal N =\{\fx^{\fv_i} \in \mathcal{M} \mid \sqrt { C_i}= (h) \}$ of monomials with `maximal' coefficient ideals
and consider the sum of `non-maximal'  coefficient ideals $\mathcal D = \sum_i C_i$,
where $\fx^{\fv_i}$ ranges  over set  $\mathcal M \setminus \mathcal N$. Notice that the monomials in $\mathcal N$ generate the  ideal $\left(\mono(J):_{R[\undl{z}]}(h)^\infty\right) \cap R $ in \autoref{conjectureColon}.

We believe that the ideal $\mathcal D$ defines the closed subset of $\mathbb{A}_\kk^{\ell u}$
that identifies the `general special' reductions $J_{\undl{\ll}}$ needed to describe the graded core. Namely, we conjecture that
\begin{enumerate}[(a)]
\item $\gradedcore(I)=J_{\undl{\ll}}^{n+1}:I^n$  for $n \gg 0$ if $\lambda$ is general in $V(\mathcal D)$.
\item $\gradedcore(I)$ can be obtained by intersecting the mono of a general minimal reduction 
with finitely many $J_{\undl{\ll}}$
with $\lambda$ general in $V(\mathcal  D)$.
\end{enumerate}
\end{discussion}

We now verify  \autoref{conjectureColon} and Conjecture (b) in   \autoref{discus_one_degree} for a specific example.

\begin{example}\label{ex3}
Let $I=(x_1^3,x_1^2x_2,x_1x_3^2,x_3^3) \subseteq R=\mathbb{Q}[x_1,x_2,x_3]$ be as in \autoref{ex1}.  
Recall that $\gradedcore (I)$ is not a finite intersection of general  reductions of $I$ and is not $\mono(K)$, for a general reduction $K$. However, as it turns out, $\gradedcore (I)$ is a finite intersection of special reductions of $I$. 

There exist relatively prime non-constant polynomials $h$ and $g$ in $\mathbb{Q}[\undl{z}]$ such that 
$$\mono(J) = (h)(x_1^2,x_1x_2,x_1x_3,x_2x_3,x_3^2)I+ (hg)(x_1^2x_2^3, x_1x_2^2x_3^2, x_2^2x_3^3)=: (h)\mathfrak{A}+(hg)\mathfrak{B}.$$
By \autoref{nonRedLocus} and \autoref{ContentIsPrincipal} the  polynomial $h$ is irreducible and defines the locus where $J_{\llb}$ is not a reduction of $I$. 
As we have seen in \autoref{ex1} 
$ \core(I) \subseteq \gradedcore (I) \subseteq \mathfrak{A}+\mathfrak{B}= I\,\fm^2$.

Since $h$ and $g$ are relative prime, we obtain $\left(\mono(J):(h)^\infty\right) \cap R=\mathfrak{A}$, with $J$ is as in \autoref{Notation3}. 
Hence $\mathfrak{A}\subseteq \gradedcore(I)$ according to \autoref{mainOneDegree}. Next we search for special reductions that are needed to compute the graded core.

Computation with Macaulay2 \cite{GS} shows that 
$$g=z_{1,4}z_{2,3}z_{3,2}-z_{1,3}z_{2,4}z_{3,2}-z_{1,4}z_{2,2}z_{3,3}+z_{1,2}z_{2,4}z_{3,3}+z_{1,3}z_{2,2}z_{3,4}-z_{1,2}z_{2,3}z_{3,4}.$$
Consider 
$$\underline{\lambda}_0=
\begin{pmatrix}
1&0&0&0\\
0&1&0&0\\
0&0&1&1
\end{pmatrix}\quad \text{and}\quad 
\underline{\lambda}_1=
\begin{pmatrix}
1&0&0&0\\
0&1&1&0\\
0&0&0&1
\end{pmatrix}.   
$$
Then $\pi_{\undl{\ll}_0}(g)=\pi_{\undl{\ll}_1}(g)=0$,  $\pi_{\undl{\ll}_0}(h)\neq 0$,  and $\pi_{\undl{\ll}_1}(h)\neq 0$. Therefore the ideals 
 $$J_{\undl{\ll}_0}=(x_1^3, x_1^2x_2, x_1x_3^2+x_3^3) \quad\text{and}\quad J_{\undl{\ll}_1}=(x_1^3, x_1^2x_2+x_1x_3^2,x_3^3)$$ 
 are special reductions of $I$. Thus 
 $$\mathfrak{A}\subseteq \gradedcore(I)\subseteq I\,\fm^2\cap \mono(J_{\undl{\ll}_0}\cap J_{\undl{\ll}_1}) =\mathfrak{A}, $$
 where $\mono(-)$ is computed using the command {\tt monomialSubideal} in Macaulay2.  
 We conclude that $\gradedcore(I)=\mathfrak{A}$, which verifies \autoref{conjectureColon}. 
\end{example}

The following theorem gives an instance where the graded core equals the mono of a general minimal reduction without any residual conditions.

\begin{theorem}\label{thm_constant_c}
Using \autoref{Notation4},  if $\sqrt{C_i}=\mathcal{C}$ for every $i$, then  $\gradedcore(I)=
\mono(J_{\undl{\lambda}})$ for general $\undl{\lambda}$. 
\end{theorem}
\begin{proof}
  The assumption implies that $\mathcal{M}=\mathcal{N}$ for  
  $\mathcal{M}$ and $\mathcal{N}$ as in \autoref{discus_one_degree}. Now we use the inclusions
$$
(\mathcal{N})=
\left(\mono(J):_{R[\undl{z}]}(h)^\infty\right) \cap R 
\subseteq \gradedcore(I)
\subseteq \mono(J_{\undl{\ll}})
= (\mathcal{M})
$$
that follow from \autoref{discus_one_degree},  \autoref{mainOneDegree}, 
and \autoref{Notation4}.
\end{proof}

\smallskip

\section{The core of lex-segment ideals}\label{lexsegments}

In this section we investigate the core of a special class of monomial ideals, lex-segment ideals. 
Throughout $R$ denotes  a polynomial ring $\kk[x_1,\ldots,x_d]$ over a field $\kk$ and $I$ denotes a homogeneous ideal. 

We begin by  
recalling some basic facts about {\it lex-segment} ideals; for a thorough treatment see \cite{MS}  or \cite{HH}. Let $H_{M}$ denote the Hilbert function of a finitely generated graded $R$-module $M$.  Write $R=\bigoplus_{i\gs 0} R_i$ and consider the lexicographic monomial order with $x_1>x_2>\cdots > x_d$.   
Let $L_i$ be the subspace of $R_i$ generated by the largest $H_{I}(i)$ monomials and set $L=\bigoplus_{i\gs 0}L_i.$ 
 The vector space $L$  is an ideal, 
  and any ideal constructed this way is called a {\it lex-segment ideal}. 
Lex-segment ideals are {\it strongly stable}, i.e., if $u\in L$ is a monomial and $x_j|u$ for some $j$, then $x_i\frac{u}{xj}\in L$ for every $i<j$. However, there are strongly stable ideals that are not lex-segment.

The purpose of this section is to tackle the following conjecture. 

\begin{conjecture}\label{conjectureLex} 
Let $R=\kk[x_1,\ldots,x_d]$ be a polynomial ring over a field $\kk$ of characteristic zero and $\fm=(x_1,\ldots, x_d)$ the homogeneous maximal ideal of $R$. If $L$ is a lex-segment ideal  of height  $g \gs 2$ generated in degree $\delta\gs 2$, then $$\core(L)=L\, \fm^{d(\delta-2)+g-\delta+1}.$$  
\end{conjecture}

\begin{remark}\label{lex conj summary}
We have strong evidence supporting this conjecture. The case $\delta=2$ was shown in \cite[Theorem 5.1]{SmCore}, and 
the  case $g=d$, i.e., $I$ is a  power of $\fm$, was shown in   
\cite[Proposition 4.2]{CPU0}. The case $d\ls 3$  is \autoref{d=3}. 
 Moreover, a large number of cases were  verified with Macaulay2 \cite{GS}.  In fact, we developed an algorithm based on \autoref{MainwithGdS2}, \autoref{ArtinN}, and \autoref{enough} that tested the conjecture for every lex-segment ideal in the following cases: 
 $d=4$ and $\delta \ls 12$; $d=5$ and $\delta \ls 5$; $d=6$ and $\delta \ls 3$. Furthermore, in \autoref{mainCore}, \autoref{firstAndLast}, and \autoref{d=g+1} we obtain other partial results towards the conjecture.
\end{remark}

For a monomial ideal $I\subset R=\kk[x_1,\ldots,x_d]$, we denote by $\Gamma(I)$ the set of monomials in $I$ and by $G(I)$ the minimal set of monomial generators $\{\fx^{\fv_1},\ldots, \fx^{\fv_u}\}$ of $I$. 
For a set of monomials $W$ in $R$, we denote by $\log (W)\subseteq \NN^d$ the set of exponents of the monomials in $W$. For $\fw=(w_1,\ldots, w_d)\in \NN^d$, we define $\min(\fw)$ and $\max(\fw)$ to be the smallest and largest $i$ such that $w_i\neq 0$, respectively; we also set $|\fw|=\sum_i w_i$.

 The following technical results are  needed in the proofs of the main results of this section.   The first one gives a characterization of the analytic spread and height of strongly stable ideals.

\begin{proposition}\label{analyticHeightSS} 	Let $R=\kk[x_1,\ldots,x_d]$ be a polynomial ring over a field $\kk$ and 
 $I$  a strongly stable ideal. 
\begin{enumerate}[$($a$)$]
\item $\height(I)=\max\{ \min(\fv)\,|\, \fv\in \log(G(I))\}.$
\item   If in addition  $I$ is generated in a single degree then  $$\ell(I)=\max\{ \max(\fv)\,|\, \fv\in \log(G(I))\}.$$
\end{enumerate}
\end{proposition}
\begin{proof}
To prove part (a) let $r=\max\{ \min(\fv)\,|\, \fv\in \log(G(I))\}$. It is clear that $I\subseteq (x_1,\ldots,x_r)$ and so $\height(I)\ls r$. On the other hand, let $\fp\in V(I)$ and let $\fv\in \log(G(I))$. If $i\ls\min(\fv)$, then $x_i^{|\fv|}\in I$ since $I$ is strongly stable. Therefore $x_i\in \fp$ for every $1\ls i\ls r$, and the conclusion follows. 

We now prove part (b).  Let $s=\max\{ \max(\fv)\,|\, \fv\in \log (G(I))\}$. Notice that $G(I)$ consists of monomials in the variables $x_1,\ldots, x_s$. Hence $\ell(I)\ls s$. 
  On the other hand, since $I$ is strongly stable and generated in one degree, say $\delta$, it follows that $x_1^{\delta-1}(x_1,\ldots, x_s)\subseteq I$. Therefore  $\ell(I)={\rm {trdeg}}_\kk(\kk[I_\delta]) \gs {\rm{trdeg}}_\kk\left(\kk[x_1^{\delta-1}x_1, \ldots, x_1^{\delta-1}x_s]\right)=s$. 
\end{proof}

\begin{remark}\label{usefulEx}
If  $L$ is a lex-segment ideal  of height  $g \gs 2$ generated in degree $\delta\gs 2$, then $\ell(L)=d$ and the minimal number of generators of $L$ is at least $d+1$. Indeed, in this case $x_2^\delta\in L$ and then $x_1^{\delta -1}(x_1,\ldots, x_d)\subset L$. The conclusion about $\ell(L)$ now follows from  \autoref{analyticHeightSS} (b).
\end{remark}

The following  proposition 
allows us to use the results of \cite{CPU} and \cite{PU1} for the computation of cores of lex-segment ideals. Some of the techniques in the proof  originate from \cite[Theorem 3.3]{SmCore}. Recall that an ideal $I$ of height $g$ is said to satisfy  $AN_{s}^{-}$, where $s$ is an integer, if for every $g \leq i \leq s$
and every   geometric $i$-residual intersection $K$ of $I$ the ring $R/K$ is Cohen-Macaulay. Notice that if $I_{\mathfrak{p}}$ satisfies $AN_{s}^{-}$ for every $\mathfrak p\in V(I)$, then $I$ satisfies $AN_{s}^{-}$.

Let $a_1, \ldots, a_n$ be homogeneous elements of $R$ and  $I$  the ideal they generate. Write $H_i$ for the  $i^{\rm{th}}$ Koszul homology of  $a_1, \ldots, a_n$. The ideal $I$ satisfies {\it{sliding depth}} if ${\rm depth} (H_i) \geq d-n+i$ for every $i$, where  we use the convention $\depth(0)=\infty$ (see \cite{HVV}).

\begin{proposition}\label{ArtinN} Let $R=\kk[x_1,\ldots,x_d]$ be a polynomial ring over a field $\kk$, $\fm=(x_1,\ldots, x_d)$ the maximal homogeneous  ideal of $R$, and $L$ a lex-segment ideal. 
Then  $L^{\rm sat}=L:\fm^{\infty}$ satisfies $G_{\infty}$, sliding depth, and  $AN_{d-1}^{-}$. Moreover, $L$ satisfies $G_{d}$ and $AN_{d-1}^{-}$.  
\end{proposition}
\begin{proof}
We may assume that $L\not= 0$ and $L\not= R$. Write $g=\height(L)$. We claim  that $L^{\rm sat}$ satisfies $G_\infty$ and sliding depth. 
Let $\delta$ be the largest degree of a monomial generator of $L$. We use induction on $\delta$. If $\delta=1$, then $L=(x_1,\ldots, x_g)$, and the claim holds trivially.  
Assume $\delta\gs 2$ and  the claim holds for every lex-segment ideal generated in degrees smaller than $\delta$.  We may assume $g<d$, as otherwise $L^{\rm sat}=R$. Set $S:=\kk[x_g,\ldots, x_d]$.  
By \autoref{analyticHeightSS} (a)  we can write $L=I+x_gL'$ for some ideals $I$ and $L'$ such that  $I\subseteq (x_1,\ldots, x_{g-1})$, the generators of $L'$ involve only the variables $x_g,\ldots, x_d$,  and $L'\cap S$ is a lex-segment ideal in $S$  generated in degrees smaller than $\delta$.  Clearly,  $x_g^\delta\in L$ and then $(x_1,\ldots, x_{g-1})\fm^{\delta-1}\subseteq I$. We conclude that
$$(x_1,\ldots, x_{g-1})\fm^{\delta-1}+x_gL'\subseteq L\subseteq (x_1,\ldots, x_{g-1})+x_gL'.$$
 Therefore 
\begin{align}
L^{\rm sat}&=(x_1,\ldots, x_{g-1})+\big((x_g L'\cap S):_S(x_g,\ldots, x_{d})^{\infty}\big)R \nonumber\\
&=(x_1,\ldots, x_{g-1})+x_g\big( (L'\cap S):_S(x_g,\ldots, x_{d})^{\infty}\big)R \text{\,\,\, since  $g<d$}.\nonumber
\end{align} 

It follows from the induction hypothesis that $L^{\rm sat}$ satisfies $G_\infty$.  Now, since $x_1,\ldots, x_{g-1}$ is a regular sequence and the image of $L^{\rm sat}$ in $S\cong R/(x_1,\ldots, x_{g-1})$ is  $x_g\big( (L' \cap S):_S(x_g,\ldots, x_{d})^{\infty}\big)$, by \cite[Lemma 3.5]{HVV} the ideal $L^{\rm sat}$ satisfies sliding depth if and only if $x_g\big( (L'\cap S):_S(x_g,\ldots, x_{d})^{\infty}\big)$  satisfies  sliding depth. Since $x_g$ is a regular element, the latter is equivalent to  $\left((L'\cap S):_S(x_g,\ldots, x_{d})^{\infty}\right)$ satisfying sliding depth. The conclusion now follows from the induction hypothesis.

Now for every $\fp \in V(L^{\rm sat})$, the ideal $L^{\rm sat}_{\fp}$ satisfies $G_{\infty}$ and sliding depth. It follows 
from \cite[Theorem~3.3]{HVV} that this ideal is $AN_{d-1}^{-}$. Hence $L^{\rm sat}$ satisfies  $AN_{d-1}^{-}$. 

Notice that the ideals $L$ and $L^{\rm sat}$ are equal locally at every prime ideal $\fp \neq \fm$. Hence the property $G_d$ passes from $L^{\rm sat}$ to $L$. 
According to \cite[Remark 1.12]{U1} the property $AN_{d-1}^{-}$ passes from $L^{\rm sat}_{\fp}$ to $L_{\fp}$ because 
the two ideals coincide locally in codimension $d-1$. Hence $L$ satisfies $AN_s^{-}$.
\end{proof}

Let $R$ be a Cohen-Macaulay $^*$local ring with a graded canonical module $\omega_R$ (cf. \cite[Section 3.6]{BH}). For a graded $R$-module $M$, we denote by  $M^{\vee}=\Hom_R(M,\,\omega_R)$ the {\it $\omega$-dual} of $M$. The following proposition and its proof are essentially contained in \cite[Lemma 2.1]{U1} (see also \cite[Lemma 4.9]{CEU}), we include it here in its graded version.

\begin{proposition}\label{propCan}
Let $R$ be a Cohen-Macaulay $^*$local ring with a graded canonical module $\omega_R$. Let $I$ be a homogeneous ideal. Let $x\in I$ be a homogeneous regular element and $J=(x): I$. Then $$\omega_{R/J}\cong \big((I\omega_{R})^{\vee \vee}/x\omega_{R}\big)(\deg(x)).$$
\end{proposition}

\begin{proof}
We may assume that $J\not=R$. There are homogeneous isomorphisms
\begin{flalign*}
 \ \ \ \ \ \  \ \ \ \ \ \ \ \ \ \  \ \ \ \ \ \ J    &= x \, (R:_{{\rm Quot}(R)}I) &&\\
		&\cong x\, \Hom_R\big(I,\,R\big)&\\
		&\cong  x\, \Hom_R\big(I,\,\Hom_R(\omega_R,\,\omega_R)\big)&&\\
		&\cong x\, \Hom_R\big(I\otimes_R\omega_R,\,\omega_R\big)&&\\ 
		&\cong x\, \Hom_R\big(I\omega_R,\,\omega_R\big), &\text{as } \Ker(I\otimes_R \omega_R \twoheadrightarrow I\omega_R) \text{ is torsion.} &
\end{flalign*}
We conclude that $J\cong  x(I\omega_R)^{\vee}$, and therefore 
\begin{equation}\label{eq_J}
J^\vee \cong\big(x(I\omega_R)^{\vee}\big)^{\vee}\cong x^{-1}(I\omega_R)^{\vee\vee}.
\end{equation} 
Dualizing the exact sequence 
\[
\begin{tikzcd}
0\arrow{r} & J\arrow{r} &R\arrow{r} &R/J\arrow{r}& 0\,,
\end{tikzcd}
\] 
into $\omega_R$, one obtains an exact sequence
\begin{equation}\label{sss_dual}
\begin{tikzcd}
	\Hom_R(R/J,\,\om_R) \arrow{r} & R^{\vee}  \arrow{r} & J^{\vee} \arrow{r}& \Ext_R^1(R/J,\,\om_R) \arrow{r}& 0\,.
\end{tikzcd}
\end{equation}
Since $\height(J)=1 \,$ we have $\Hom_R(R/J,\om_R)=0$ and $\Ext_R^1(R/J,\,\om_R) \cong \omega_{R/J}$. Thus  
\autoref{eq_J} and \autoref{sss_dual}  yield
\[
\begin{tikzcd}
 0 \arrow{r} & \om_R \arrow{r}{} & x^{-1}(I\omega_R)^{\vee\vee} \arrow{r} & \om_{R/J} \arrow{r} & 0\,.
\end{tikzcd}
\] 
Hence
 $$\om_{R/J}\cong \pars{x^{-1}(I\omega_R)^{\vee\vee}}/\om_R \cong \left((I\omega_{R})^{\vee \vee}/x\omega_{R}\right)\left(\deg(x)\right)\,,$$ as desired.
\end{proof}

\smallskip

For a graded module $M=\oplus_{i\in \ZZ}M_i$ we  denote by $\indeg(M)$ the {\it initial degree} of $M$, i.e., $\indeg(M)=\inf\{i\,|\, M_i\neq 0\}.$

\smallskip

\begin{lemma}\label{equivs}
Let $R$ be a standard graded Cohen-Macaulay ring over a field $\, \kk$ with  $\dim(R)=d$, $\omega_R$ the graded canonical module of $R$, and $I$ a homogeneous ideal. Assume that $I$ satisfies $G_{d-1}$ and is weakly $(d-2)$-residually $S_2$. Let $n$ and $\delta \gs 0$   be integers, and consider the following statements $:$
\begin{enumerate}[{\rm (i)}]
\item $\indeg \pars{\omega_{R/\pars{(a_1,\ldots,a_{d-1}): I}}}\gs -n \, $ for some $(d-1)$-residual intersection $$(a_1,\ldots,a_{d-1}):I$$  of $I$ such that each $a_i$ is homogeneous of degree $\delta$,
\item  $\indeg\pars{\omega_{R/\pars{(a_1,\ldots,a_{d-1}): I}}}\gs -n \, $ for every $(d-1)$-residual intersection  $$(a_1,\ldots,a_{d-1}):I$$   of $I$ such that each $a_i$ is homogeneous of degree $\delta$,
\item $\indeg\pars{\omega_{R/\pars{(a_1,\ldots,a_{d}): I}}}\gs -n \, $ for every $d$-residual intersection $$(a_1,\ldots,a_{d}):I$$  of $I$ such that each $a_i$ is homogeneous of degree $\delta$.
\end{enumerate}
Then \tn{(i)}  is equivalent to  \tn{(ii)}. Moreover, if $\, \indeg( I)\gs \delta$, then \tn{(ii)} implies  \tn{(iii)}.

\end{lemma}
\begin{proof} 
For a Noetherian graded $\kk$-algebra $T$, we denote by ${\rm HS}_{T}(t)$ the Hilbert series of $T$. We may assume that the field $\kk$ is infinite. 

(i) $\Rightarrow$ (ii): Set $\underline{a}=\{a_1,\ldots, a_{d-1}\}$ and let $\overline{R}=R/\pars{(\underline{a}):I}$. 
Since $\overline{R}$ is Cohen-Macaulay of dimension one  (see \cite[Proposition 3.4 (a)]{CEU}), we may write $\HS_{\overline{R}}(t)=\frac{Q_{\overline{R}}(t)}{1-t}$ for some $Q_{\overline{R}}(t)\in \ZZ[t]$.
 By \cite[Corollary~4.4.6~(a)]{BH} and the assumption in (i), we have  
 \begin{equation}\label{indeg_omega}
 	\deg\big(Q_{\overline{R}}(t)\big)=1-\indeg \big(\omega_{\overline{R}}\big)\ls 1+n.
 \end{equation}
By \cite[Proposition 3.1 and Theorem 2.1 (b)]{CEU}, $\HS_{R/((\underline{a}):I)}(t)$ is the same for every $(d-1)$-residual intersection $\pars{(\underline{a}):I}$ of $I$ such that each $a_i$ is homogeneous of degree $\delta$. The conclusion now follows by applying 
\autoref{indeg_omega} again.

(ii) $\Rightarrow$ (iii): We assume that $\indeg( I)\gs \delta$. Let $(\underline{a}, a_{d}):I=(a_1,\ldots, a_{d-1}, a_{d}):I$ be a $d$-residual intersection of $I$. By \cite[Corollary 1.6 (a)]{U1} we may assume that $(\underline{a}):I$ is a geometric $(d-1)$-residual intersection of  $I$.
 Write $\overline{\phantom{M}}$ for images in  $\overline{R}=R/\left((\underline{a}):I\right)$.
  From \cite[Proposition 3.1, Proposition 3.3, and Lemma 2.4 (b)]{CEU} it follows that 
$\overline{R}$ is Cohen-Macaulay of dimension one, $\overline{a_d}\in \overline{I}$ is a homogeneous $\overline{R}$-regular element of degree $\delta$, 
and $\overline{(\underline{a},\, a_d):_R I}=\overline{(a_d)}:_{\overline{R}}\overline{I}.$ 
Hence by \autoref{propCan} and the fact that $I\omega_{\overline{R}}$ is a maximal Cohen-Macaulay $\overline{R}$-module, we have 
$$\omega_{R/\pars{(\underline{a},\, a_d):I}}
=\omega_{\overline{R}/\pars{(\overline{a_d}):\overline{I}}}\cong \left((I\omega_{\overline{R}})^{\vee\vee}/a_d\omega_{\overline{R}}\right) (\delta)\cong \left(I\omega_{\overline{R}}/a_d\omega_{\overline{R}}\right)(\delta),$$
where $(-)^{\vee}=\Hom_{\overline{R}}\left(-,\omega_{\overline{R}}\right).$
Therefore, $\indeg\left(\omega_{R/\left((\underline{a},\, a_d):I\right)}\right)\gs -n$ as desired. 
\end{proof}

\smallskip

\begin{remark}\label{onevarless}
Let $R=\kk[x_1,\ldots,x_d]$ be a polynomial ring over a field $\kk$ and  $L$ a lex-segment ideal of height $g$ generated in degree $\delta$. Let $R'=\kk[x_1,\ldots, x_{d-1}]$. Then $L\cap R'$ is a lex-segment ideal of $R'$ generated in degree $\delta$ and $\height\left(L'\right)=\min\{g,d-1\}$. 
\end{remark}
\begin{proof}
 If $g=d$, then $L=(x_1,\ldots, x_d)^\delta$ and the result is clear. Hence we may assume $g<d$. Let $\{\fx^{\fv_1},\ldots, \fx^{\fv_u}\}$ be the minimal monomial generating set of $L$. Thus  $L\cap R'$ is generated by the 
 monomials $\fx^{\fv_i}$ such that $x_d\nmid \fx^{\fv_i}$, and then it is a lex-segment ideal of $R'$. Finally, by \autoref{analyticHeightSS} (a) we have $\height(L')=g$. 
\end{proof}

In the following we  prove one inclusion of \autoref{conjectureLex} in full generality. 

\begin{theorem}\label{mainCore}
Let $R=\kk[x_1,\ldots,x_d]$ be a polynomial ring over an infinite field $\, \kk$ and $\fm=(x_1,\ldots, x_d)$ the  maximal homogeneous ideal of $R$. 
If $L$ is a lex-segment ideal of height $g\gs 2$ generated in degree $\delta\gs 2$, then $$ L\, \fm^{d(\delta-2)+g-\delta+1}\subseteq \core(L).$$  
\end{theorem}
\begin{proof} 
Recall that by \autoref{ArtinN} the ideal $L$ satisfies $G_{d}$ and $AN_{d-1}^-$.
According to \autoref{remCo}, we have that $\core(L)$ is the intersection of 
finitely many reductions generated by $d$ general elements of $L$ with respect to a generating set of $L$ contained in $L_\delta$. Let $\underline{a}=a_1,\ldots, a_d$ be such  general elements. 
To prove the statement of the theorem, it suffices to show that 
 $\fm^{d(\delta-2)+g-\delta+1}\subseteq (\underline{a}):L\,$.
 The latter is equivalent to 
\begin{equation}\label{feq}  H_{R/\left((\underline{a}):L\right)}(n)=0\,\,\,\,\,\text{ for every  }\,\,n\gs d(\delta-2)+g-\delta+1.
\end{equation}  
Since $L$ is $G_{d}$, by \cite[Lemma 3.1 (a)]{PX} and \autoref{usefulEx},  $(\underline{a}):L$ is a $d$-residual intersection of $L$. Therefore the ring $R/\pars{(\underline{a}):L}$ is Artinian. Hence  \eqref{feq} is equivalent to $\indeg(\omega_{R/\left((\underline{a}):L\right)})\gs -\left(d(\delta-2)+g-\delta\right)$. 

We claim that there exists a $(d-1)$-residual intersection of $L$, $(b_1,\ldots, b_{d-1}):L$, such that each $b_i$ is homogeneous of degree $\delta$ and $\indeg \left(\omega_{R/\left((b_1,\ldots, b_{d-1}):L\right)}\right)\gs -\big(d(\delta-2)+g-\delta\big)$. The result will follow from this claim and the implication (i) $\Rightarrow$ (iii) in \autoref{equivs}. 

We now prove the claim by induction on $\sigma(L)=d-g\gs 0$. 
If $\sigma(L)=0$, then $L=\fm^\delta$  and the claim is satisfied by taking $b_i=x_i^\delta$ for every $i$ (see \cite[Corollary 3.6.14]{BH}). 

For the induction step assume $\sigma(L)>0$ and set $R'=\kk[x_1,\ldots, x_{d-1}]$ and $L'=L\cap R'$. By \autoref{onevarless} the ideal $L'$ is a lex-segment ideal  of height $g$ generated in degree $\delta$. In particular, $\sigma(L')=d-1-g<\sigma(L)$. 
By induction hypothesis there exists a $(d-2)$-residual intersection    
$(\underline{b}):_{R'}L'=(b_1,\ldots, b_{d-2}):_{R'}L'$ such that each $b_i$ is homogeneous of degree $\delta$ and
\begin{equation}\label{eq_IP}
\indeg\left(\omega_{R'/\pars{(\underline{b}):_{R'}L'}}\right)\gs -\left((d-1)(\delta-2)+g-\delta\right)\,.
\end{equation} 
Therefore the implication (i) $\Rightarrow$ (ii) in \autoref{equivs} shows that every $(d-2)$-residual intersection  
$(\underline{b}):_{R'}L'$ such that each $b_i$ is homogeneous of degree $\delta$
 has this property. 
Hence 
we  may choose this residual intersection to be a geometric residual intersection, which exists  
 by \autoref{ArtinN} and \cite[proof of Lemma 1.4]{U1}. 

To pass back to $L$ we consider the saturation $L^{\rm sat}$ and write $(L^{\rm{sat}})'=L^{\rm{sat}}\cap R'$. From \cite[Proposition 15.24]{E} we have  $L^{\rm{sat}}=L:(x_d)^{\infty}$, hence the minimal monomial generators of $L^{\rm{sat}}$ are not divisible by $x_d$, i.e.,
\begin{equation}\label{sat1}L^{\rm{sat}}=(L^{\rm{sat}})'R\, .
\end{equation}

By \autoref{ArtinN},  $L^{\rm{sat}}$ satisfies $G_{\infty}$ and $AN_{d-1}^-$, hence so does $(L^{\rm{sat}})'$. 
The homogeneous inclusion $(L^{\rm{sat}})'/L' \hookrightarrow L^{\rm{sat}}/L$ 
implies that the Hilbert function of $(L^{\rm{sat}})'/L'$ is eventually zero, hence 
\begin{equation}\label{htU}
    \height \left(L':_{R'}(L^{\rm{sat}})'\right)\gs d-1\, .
\end{equation} 
Notice that $(\underline{b}):_{R'}(L^{\rm{sat}})' \subset  ( \underline{b}):_{R'}L'$. Since $( \underline{b}):_{R'}L'$ 
is a geometric $(d-2)$-residual intersection, \autoref{htU} implies that $( \underline{b}):_{R'}(L^{\rm{sat}})'$ is a geometric $(d-2)$-residual intersection. The  ideal  $( \underline{b}):_{R'}(L^{\rm{sat}})'$ is  unmixed of height $d-2$ by  \cite[Proposition 1.7 (a)]{U1}. 
 Therefore by \autoref{htU} 
  \begin{equation}\label{equalyes}
( \underline{b}):_{R'}(L^{\rm{sat}})'=( \underline{b}):_{R'}L'. 
\end{equation}

Let $\overline{\phantom{M}}$ denote images in the ring 
	$\overline{R'}=R'/\left(( \underline{b}):_{R'}(L^{\rm{sat}})'\right)=R'/\left(( \underline{b}):_{R'}L'\right)$. 
	  Notice that $\overline{R'}$ is Cohen-Macaulay of dimension one. 
Since $(\underline{b}):_{R'}L'$ is a geometric $(d-2)$-residual intersection of $L'$ we have that $\height\left(\overline{L'}\right)=1$. Hence there exists an   $\overline{R'}$-regular element 
$b_{d-1}\in L'_{\delta}$. 
Thus by \autoref{equalyes} the ideal 
$(\underline{b},\,b_{d-1}):_{R'}(L^{\rm{sat}})'$ 
is a $(d-1)$-residual intersection. 
  From \cite[Proposition 1.7 (f)]{U1} 
it follows that 
 $\overline{(\underline{b},\,b_{d-1}):_{R'}(L^{\rm{sat}})'}=\overline{(b_{d-1})}:_{\overline{R'}} \overline{(L^{\rm{sat}})'}$.  
Moreover, $(L^{\rm{sat}})'\omega_{\overline{R'}}$ is a maximal Cohen-Macaulay $\overline{R'}$-module. 
Hence \autoref{propCan} implies
$$\omega_{R'/\pars{(\underline{b},\, b_{d-1}):_{R'}(L^{\rm{sat}})'}}\cong \left((L^{\rm{sat}})'\omega_{\overline{R'}}/b_{d-1}\omega_{\overline{R'}}\right)(\delta).$$ 
Using \autoref{sat1} we see that 
$$\omega_{R/\pars{(\underline{b},\,b_{d-1}):_R L^{\rm{sat}}}} \cong
\omega_{R'/\pars{(\underline{b},\, b_{d-1}):_{R'}(L^{\rm{sat}})'}}\otimes_{R'} R\,(-1)\, .$$
Hence
\begin{align*} 
\indeg \left(\omega_{R/\pars{(\underline{b},\,b_{d-1}):_R L^{\rm{sat}}}}\right)&=
\indeg \left(\omega_{R'/\pars{(\underline{b},\, b_{d-1}):_{R'}(L^{\rm{sat}})'}}\right)+1\\
&=
\indeg \left(\left((L^{\rm{sat}})'\omega_{\overline{R'}}/b_{d-1}\omega_{\overline{R'}}\right)(\delta)\right)+1& \\
 &\gs 
 \indeg \left(\omega_{\overline{R'}}\right) -\delta+2&\\
       &\gs -\left((d-1)(\delta-2)+g-\delta\right)-\delta+2& \text{  by \autoref{eq_IP}}\\ 
       &=-\left(d(\delta-2)+g-\delta\right).&
\end{align*}

Finally, since $(\underline{b},\,b_{d-1}):_{R'}(L^{\rm{sat}})'$ 
is a $(d-1)$-residual intersection, \autoref{sat1} shows that $(\underline{b},b_{d-1}):_RL^{\rm{sat}}$ is  a $(d-1)$-residual intersection. 
 This ideal is unmixed of height $d-1$ by \cite[Proposition 1.7 (a)]{U1}, and moreover $\height \left(L:_RL^{\rm{sat}}\right)\gs d$.
Therefore $
( \underline{b},b_{d-1}):_RL^{\rm{sat}}=( \underline{b},b_{d-1}):_RL
\, $  is a $(d-1)$-residual intersection of $L$, completing the proof. 
\end{proof}

\begin{remark}\label{enough}
We remark that in order to show the reverse containment in \autoref{mainCore}, if  $\kk$ is a field of characteristic zero, it is enough to show that $x_1^{d(\delta-2)+g}\not\in J$ for some reduction   $J$ of $L$. To see this, 
notice that $L\cap \fm^{d(\delta-2)+g+1}=L\,\fm^{d(\delta-2)+g-\delta+1}$ because $L$ is generated in degree $\delta$.
Therefore it suffices to show that $\core(L)\subseteq \fm^{d(\delta-2)+g+1}$. Since $\core(L)$ is a strongly stable monomial ideal \cite[Proposition 2.3]{SmCore},  this containment is equivalent to $x_1^{d(\delta-2)+g}\not\in \core(L)$. 
\end{remark}

The next  two theorems settle \autoref{conjectureLex}  in some particular cases. In the first  theorem, we show that  \autoref{conjectureLex} holds for the smallest  and largest lex-segment ideals for fixed $d$, $g$, and $\delta$.

\begin{theorem}\label{firstAndLast} 
Let $R=\kk[x_1,\ldots,x_d]$ be a polynomial ring over a field $\kk$ of characteristic zero and $\fm=(x_1,\ldots, x_d)$ the  maximal homogeneous ideal of $R$. 
Let $g$ and $\delta$ be integers  such that $2\ls g\ls d$ and $\delta\gs 2$.  Let $L$ be one of the following  lex-segment ideals
\begin{enumerate}[$(1)$]
\item $(x_1,\ldots, x_{g-1})\fm^{\delta-1}+(x_g)^\delta$, or
\item $(x_1,\ldots, x_g)\fm^{\delta-1}$.
\end{enumerate}
Then $\core(L)=L\, \fm^{d(\delta-2)+g-\delta+1}.$
\end{theorem}

We need the following lemma for the proof of \autoref{firstAndLast}.\

\begin{lemma}\label{core in formula} 
Let $R=\kk[x_1,\ldots,x_d]$ be a polynomial ring over a field $\kk$ of characteristic zero and  $L$ a lex-segment ideal generated in degree $\delta\gs 2$.  
If $J$ is any  reduction of $L$, then for every $n \gs 0$ we have
$$\core(L) \subseteq J^{n+1}:L^n.$$
\end{lemma}

\begin{proof} 
Since $L$ satisfies $G_{d}$ and $AN_{d-1}^-$ according to \autoref{ArtinN},  by \autoref{corelocalizes}  it suffices to show that $\core(L_{\fp}) \subseteq K^{n+1}:_{R_\fp}L_{\fp}^n$ for every $\fp\in {\rm Spec} (R)$ and every reduction $K$ of $L_{\fp}$.  
We may further assume that $K$ is a minimal reduction of $L_{\fp}$, and in particular $\mu(K)\ls \dim(R_\fp)$. 
By 
 \cite[Lemma 1.10 (b)]{U1}  the ideal 
 $L_{\fp}$ satisfies $G_{d}$ and $AN_{d-1}^-$, and by \cite[Remark 2.7]{JU} we have 
 $\height\left(K:_{R_\fp}L_{\fp}\right)\gs \dim(R_\fp)$. Therefore $K$ satisfies $G_{\infty}$. Hence by \cite[Remark 1.12 and Corollary 1.8 (c)]{U1}, $K$ satisfies sliding depth. 
 Now the proof of \cite[Theorem 4.4]{PU1} shows that for all $n\gs 0$
 $$\core(L_{\fp}) \subseteq K^{n+1}:_{R_\fp} \sum \limits_{y\in L_{\fp}}(K,y)^n=K^{n+1}:_{R_\fp}L_{\fp}^{n}\,,$$
where the last equality holds since $\kk$ has characteristic zero  and then $L_{\fp}^{n}=\sum \limits_{y\in L_{\fp}}(y^n)$.
\end{proof}

\begin{proof}[Proof of \autoref{firstAndLast}]
We write $L_1=(x_1,\ldots, x_{g-1})\fm^{\delta-1}+(x_g^\delta)$ and $L_2=(x_1,\ldots, x_g)\fm^{\delta-1}$. Let $$J_1=(x_1^\delta,\ldots, x_g^\delta)+(x_1,\ldots, x_{g-1})(x_{g+1}^{\delta-1},\ldots, x_d^{\delta-1}),\text{ and }$$ $$J_2=(x_1^\delta,\ldots, x_g^\delta)+(x_1,\ldots, x_{g})(x_{g+1}^{\delta-1},\ldots, x_d^{\delta-1}).$$ We claim that $J_1$ is a reduction of $L_1$ and $J_2$ is a reduction of $L_2$. To see this, notice that by \cite[Proposition~8.1.7]{HS} the ideal $(x_1,\ldots, x_{g-1})(x_{1}^{\delta-1},\ldots, x_d^{\delta-1})+(x_g^\delta)$ is a reduction of $L_1$, and this ideal is equal to
$$
(x_1,\ldots, x_{g-1})(x_{1}^{\delta-1},\ldots, x_g^{\delta-1})+(x_g^\delta) + (x_1,\ldots, x_{g-1})(x_{g+1}^{\delta-1},\ldots, x_d^{\delta-1}).
$$
Clearly the ideal $(x_1^\delta,\ldots, x_g^\delta)$ is a reduction of $(x_1,\ldots, x_{g-1})(x_{1}^{\delta-1},\ldots, x_g^{\delta-1})+(x_g^\delta)\subset (x_1, \ldots, x_g)^\delta$. 
Therefore again by  \cite[Proposition 8.1.7]{HS} and transitivity of reductions we conclude $J_1$ is a reduction of $L_1$.  Likewise, $J_2$ is a reduction of $L_2$.

Now by \autoref{enough} and \autoref{core in formula}, it suffices to show that $x_1^{d(\delta-2)+g} \notin J_1^d:L_1^{d-1}$  and $x_1^{d(\delta-2)+g} \notin J_2^d:L_2^{d-1}$.  
Let $\alpha=x_1^{2d-g-1}x_2^{\delta-1}\cdots x_g^{\delta-1}x_{g+1}^{\delta-2}\cdots x_d^{\delta-2}$ 
and notice that $\alpha\in L_1^{d-1}\subseteq  L_2^{d-1}$. We now show that $x_1^{d(\delta-2)+g}\alpha\not\in J_2^d$, and hence $x_1^{d(\delta-2)+g}\alpha \not \in J_1^d$, which  finishes the proof. Indeed, suppose by contradiction that  
$$\beta:=x_1^{d(\delta-2)+g}\alpha=x_1^{d\delta-1}x_2^{\delta-1}\cdots x_g^{\delta-1}x_{g+1}^{\delta-2}\cdots x_d^{\delta-2}\in J_2^d.$$
Since  none of the minimal monomial generators of $J_2$, other than $x_1^\delta$, divides $\beta$, we must have $x_1^{d\delta}$ divides  $\beta$, a contradiction.
\end{proof}

The next theorem shows that \autoref{conjectureLex} holds for any $\delta$ if $g=d-1 \geq 2$. 

\begin{theorem}\label{d=g+1} Let $R=\kk[x_1,\ldots,x_d]$ be a polynomial ring over a field $\kk$ of characteristic zero, $\fm=(x_1,\ldots, x_d)$ the  maximal homogeneous ideal of $R$, and $L$ a lex-segment ideal generated in degree $\delta\gs 2$. Assume $d \geq 3$ and that $L$ has height $g=d-1$. Then $$\core(L)=L\, \fm^{d(\delta-2)+g-\delta+1}.$$
\end{theorem}
\begin{proof}
By \autoref{enough} it suffices to show that $x_1^{d(\delta-1)-1}\not\in \core(L)$. Thus  by  \autoref{core in formula}, it is enough to prove that $x_1^{d(\delta-1)-1} \not\in J^d:L^{d-1}$ for some  reduction $J$ of $L$. 
Since $g=d-1$, it follows that  $L=(x_1, \ldots, x_{d-2})\fm^{\delta-1}+x_{d-1}L'$, where $L'$ is a lex-segment ideal  in the variables $x_{d-1}$ and $x_d$ generated in degree $\delta-1$.  
By \autoref{firstAndLast} we may assume that 
$$L=(x_1, \ldots, x_{d-2})\fm^{\delta-1}+\left(x_{d-1}^{\delta}, x_{d-1}^{\delta-1}x_d, \ldots, x_{d-1}^{\delta-i}x_d^i\right)$$ with $1\ls i \ls \delta-2$.  Therefore $K=\left(x_1^{\delta}, \ldots, x_{d-1}^\delta, x_{d-1}^{\delta-i}x_d^i\right)+(x_1, \ldots, x_{d-2})x_d^{\delta-1}$ is a  reduction of $L$ according to \cite[Proposition 2.1]{Sin}.

We claim that $$J=\left(x_1^\delta-x_{d-1}^{\delta-i}x_d^i, x_2^\delta, \ldots, x_{d-1}^\delta\right)+(x_1, \ldots, x_{d-2})x_d^{\delta-1}$$ is a reduction of $L$ and that $x_1^{d(\delta-1)-1} \not\in J^d:L^{d-1}$. 

First we  show that $J$ is a reduction of $L$. 
Let $$K'=(x_1^\delta, x_1x_d^{\delta-1},  x_{d-1}^\delta, x_{d-1}^{\delta-i}x_d^i,)
\qquad \text{and} \qquad 
J'=\left(x_1^\delta-x_{d-1}^{\delta-i}x_d^i,  x_1x_d^{\delta-1}, x_{d-1}^\delta\right).$$
 Let $\mathcal{F}(K')$ denote the special fiber ring of $K'$. Then $\mathcal{F}(K')\cong \kk[T_1, T_2, T_3, T_4]/\mathcal{D}$, for some homogeneous ideal $\mathcal{D}$, where the isomorphism is induced by the map sending $T_1$ to $x_1^\delta$, $T_2$ to $x_1x_d^{\delta-1}$, $T_3$ to $x_{d-1}^\delta$,  and $T_4$ to   $x_{d-1}^{\delta-i}x_d^i$. 
 Notice that $T_1^iT_4^{\delta(\delta-1)}-T_2^{\delta i}T_3^{(\delta-1)(\delta-i)}\in \mathcal{D}$. Therefore the ring $$\mathcal{F}(K')/J'\mathcal{F}(K')\cong \mathcal{F}(K')/(T_1-T_4,T_2,T_3)\mathcal{F}(K')$$  is Artinian; here  
 we denote by $J'\mathcal{F}(K')$ the $\mathcal{F}(K')$-ideal generated by the image of $J'$ in $[\mathcal{F}(K')]_1$.   
Thus $J'$ is a reduction of $K'$, and hence $J$ is a reduction of $K$.  We conclude that  $J$ is a reduction of $L$, proving the claim.

Next we show  $x_1^{d(\delta-1)-1} \not\in J^d:L^{d-1}$. Let $\alpha=x_1^dx_2^{\delta-1} \cdots x_{d-1}^{\delta-1}x_d^{\delta-2}$ and notice that  $\alpha \in L^{d-1}$. We claim that $\alpha x_1^{d(\delta-1)-1} \not\in J^d$,  
which will complete  the proof.

Let 
$$H=\left(x_1^\delta-x_{d-1}^{\delta-i}x_d^i,\, x_{d-1}^\delta\right)+(x_1, \ldots, x_{d-2})x_d^{\delta-1} \subset J.$$ 
Clearly $\alpha x_1^{d(\delta-1)-1} \in J^d$ if and only if $\alpha x_1^{d(\delta-1)-1} \in H^{d}$. Thus we focus the rest of the proof on showing that  $\alpha x_1^{d(\delta-1)-1} \not\in H^{d}$.  For this we consider  the sequence $\undl{f}=f_1,\, f_2,\, f_3$, where $$f_1=x_1^{2d\delta-\delta},\quad f_2= x_{d-1}^{d\delta},\quad f_3=x_d^{d(\delta-1)}\, ,$$ 
and note that it suffices to show 
$$C:=(\underline{f}):\alpha x_1^{d(\delta-1)-1}\not\supseteq (\underline{f}):H^d.$$

 It is easy to see that  
$$C=\left(x_1^{(d-1)\delta+1},\, x_{d-1}^{(d-1)\delta+1},\, x_d^{(d-1)(\delta-1)+1}\right).$$ 
Consider the element
$$\beta:=\left(x_1^\delta+dx_{d-1}^{\delta-i}x_d^{i}\right)M\,, \quad\text{ where }\quad M=x_1^{(d-2)\delta}x_{d-1}^{(d-1)\delta}x_d^{(d-1)(\delta-1)}.$$ 
Since 
$$x_1^\delta M=x_1^{(d-1)\delta}x_{d-1}^{(d-1)\delta}x_d^{(d-1)(\delta-1)} \not\in C,$$  
we have that  $\beta \not\in C$.   
Thus it is enough to show $\beta \in (\underline{f}): H^d$. Since $x_{d-1}^{\delta}M\in (\underline{f})$ and $ x_d^{\delta-1}M \in (\underline{f})$, it remains  to see that $\beta \left(x_1^\delta-x_{d-1}^{\delta-i}x_d^i\right)^d \in (\underline{f})$.

Notice that 
\begin{eqnarray*}
\beta (x_1^\delta-x_{d-1}^{\delta-i}x_d^i)^d&=&\beta \sum\limits_{j=0}^{d}(-1)^j \binom{d}{j} x_1^{\delta (d-j)}x_{d-1}^{(\delta-i)j}x_d^{ij}\\
&=& x_1^{(d-1)\delta}x_{d-1}^{(d-1)\delta}x_d^{(d-1)(\delta-1)}  \sum\limits_{j=0}^{d}(-1)^j \binom{d}{j} x_1^{\delta (d-j)}x_{d-1}^{(\delta-i)j}x_d^{ij}\\
&&\,\,+\,\,dx_1^{d\delta-2\delta}x_{d-1}^{d\delta-i}x_d^{(d-1)(\delta-1)+i} \sum\limits_{j=0}^{d}(-1)^j \binom{d}{j} x_1^{\delta (d-j)}x_{d-1}^{(\delta-i)j}x_d^{ij}\,.
\end{eqnarray*}
To simplify the notation, set $$\beta_1:=x_1^{(d-1)\delta}x_{d-1}^{(d-1)\delta}x_d^{(d-1)(\delta-1)} \qquad\text{and}\qquad \beta_2:=dx_1^{d\delta-2\delta}x_{d-1}^{d\delta-i}x_d^{(d-1)(\delta-1)+i}\, ,$$ 
and for  $\,0\leq j \leq d$ set 
$$h_j:=(-1)^j \binom{d}{j} x_1^{\delta(d-j)}x_{d-1}^{(\delta-i)j}x_d^{ij}\,.$$ 
It is easy to see that $\beta_1h_0 \in (\underline{f})$ and $\beta_1h_1+\beta_2h_0=0$.  We prove below that  $\beta_1 h_j \in (\underline{f})$ for every $j \geq 2$ and that  $\beta_2 h_j \in (\underline{f})$ for every $j \geq 1$.

Consider the terms $\beta_1 h_j$ with $j \geq 2$. In $\beta_1h_j$  the degree of $x_{d-1}$ is $(d-1)\delta+(\delta-i)j$ and the degree of $x_d$ is $(d-1)(\delta-1)+ij$. Hence to show that $\beta_1h_j \in (\underline{f})$ we need to prove that for every $2\ls j\ls d$ one of the following inequalities holds
\begin{equation}\label{eq_first_j}
(d-1)\delta+(\delta-i)j \geq d\delta 
\qquad\text{ or }\qquad 
(d-1)(\delta-1)+ij \geq d(\delta-1).
\end{equation}
The first inequality is equivalent to $i \leq \frac{\delta(j-1)}{j}$ and the second one is equivalent to $i\geq \frac{\delta-1}{j}$. Since $ \frac{\delta(j-1)}{j} \geq \frac{\delta-1}{j}$ for $j \geq 2$, we have that one of the inequalities in \autoref{eq_first_j} must hold for any such $j$.

Finally, consider the terms $\beta_2h_j$, with $j \geq 1$. In $\beta_2h_j$  the degree of $x_{d-1}$ is $d\delta-i+(\delta-i)j$ and the degree of $x_d$ is $(d-1)(\delta-1)+i+ij$.  
Hence, to show that $\beta_2h_j \in (\underline{f})$ we need to prove that for every $1\ls j\ls d$ one of the following inequalities holds 
\begin{equation}\label{eq_second_j}
	d\delta-i+(\delta-i)j\geq d\delta 
	\qquad\text{ or }\qquad 
(d-1)(\delta-1)+i+ij \geq d(\delta-1).
\end{equation}
The first inequality is equivalent to $i \leq  \frac{\delta j}{j+1}$ and the second one is equivalent to $i\geq \frac{\delta-1}{j+1}$. Since $\frac{\delta j}{j+1} \geq \frac{\delta-1}{j+1}$ for $j \geq 1$, we have that one of the inequalities in \autoref{eq_second_j} must hold for any such $j$. 
\end{proof}

\begin{corollary}\label{d=3}
\autoref{conjectureLex} holds if  $d\ls 3$.
\end{corollary}
\begin{proof}
The result follows from  \autoref{d=g+1} and   \cite[Proposition 4.2]{CPU0}.
\end{proof}

The next fact follows directly by combining several  results in the literature. We state it here for completeness and to provide a reference. For more information about integral closures and reductions see \cite{HS} .  

\begin{proposition}\label{LisNormal}
Let $R=\kk[x_1,\ldots,x_d]$ be a polynomial ring over a field $\kk$ 
and $L$ a lex-segment ideal generated in degree $\delta \gs 1$.  The ideal $L$ is normal, i.e., $L^n$ is integrally closed for every $n\in  \NN$.
\end{proposition}
\begin{proof}
The proof follows from \cite[Theorem 5.1]{HHVl}, \cite[Proposition 2.14 and its proof]{DeNegri}, and \cite[Proposition 13.15]{Sturm}.
\end{proof}

The following two results provide upper bounds for the core of lex-segment ideals generated in a single degree.

\begin{theorem}\label{core_in_adj}
Let $R=\kk[x_1,\ldots,x_d]$ be a polynomial ring over a field $\kk$ of characteristic zero.  If $L$ is a lex-segment ideal of height $g\gs 2$ generated in degree $\delta\gs 2$, then $$\core(L) \subseteq {\rm adj}(L^{g}),$$ where ${\rm adj} (L^g)$ denotes the adjoint of $L^g$ as in \cite[Definition 1.1]{L}.
\end{theorem}

\begin{proof}
 Let $J$ be any minimal reduction of $L$. Let  $A=R[Jt,t^{-1}]$ and  $B=R[Lt,t^{-1}]$ be the extended Rees algebras  of $J$ and $L$, and 
  $\omega_A$ and $\omega_B$ be  their graded canonical modules. 
   Notice that
   $$\omega_{A}\subset 
   \left(\omega_{A}\right)_{t^{-1}}
   =   
   \omega_{A_{t^{-1}}}
   =
   \omega_{R[t,t^{-1}]}\,.$$
Thus making the identification  $\omega_{R[t,t^{-1}]}=R[t,t^{-1}]$ we obtain an embedding 
$\omega_A\subset R[t,t^{-1}]$ so that 
$\left(\omega_{A}\right)_{t^{-1}}=R[t,t^{-1}]$.  
  Therefore  $[\omega_{A}]_nt^{-n}=R$ for $n$ sufficiently small.     
  
 By \autoref{LisNormal}  $L$ is a normal monomial ideal, thus   $B$ is a normal Cohen-Macaulay algebra and a direct summand of a polynomial ring according to \cite[Theorem 6.10 and Theorem 4.43]{BG}. Therefore, $B$ has rational singularities \cite[Th\'eor\`eme]{Boutot}. 
 We conclude that 
 \begin{equation}\label{eq_adj}
[\omega_B]_{i}t^{-i}={\rm adj}(L^i)\,.
 \end{equation}
 for every $i\gs 0$ by \cite[proof of Corollary 3.5]{Hury}. Let $K:=\Quot(R)$ be the field of fractions of $R$ and let $r:=r_J(L) $ be the reduction number of $L$ with respect to $J$. 
 We have the following  isomorphisms of graded $A$-modules
 \begin{align*}
 	\omega_{B} \cong \Hom_A(B,\, \omega_{A})&\cong \omega_A:_{K(t)} B\\
 	&=\omega_A:_{R[t,t^{-1}]}B\\
 	&= \omega_A:_{R[t,t^{-1}]}(R \oplus Lt \oplus \ldots \oplus L^{r}t^{r})\\
 	&= \bigcap \limits_{i=0}^{r}\left( \omega_A:_{R[t,t^{-1}]}L^{i}\right)t^{-i}\\
 	&=  \bigcap \limits_{i=0}^{r}\left( \pars{\bigoplus \limits_{j\in \ZZ}[\omega_A]_j}:_{R[t,t^{-1}]}L^{i}\right)t^{-i}\\
 	&= \bigcap \limits_{i=0}^{r} \bigoplus \limits_{j\in \ZZ}\left([\omega_A]_jt^{-j}:_{R}L^{i}\right)t^{j-i}\\
 	&= \bigoplus \limits_{s \in \mathbb{Z}}\left( \bigcap \limits_{i=0}^{r}\left( [\omega_A]_{s+i}t^{-s-i}:_{R}L^{i}\right)\right)t^{s}.
 \end{align*}
 In particular, 
 \begin{equation}\label{wB}
  [\omega_B]_gt^{-g}=\bigcap \limits_{i=0}^{r}\left( [\omega_A]_{g+i}t^{-g-i}:_{R}L^{i}\right)\,.  
 \end{equation}

 By \autoref{ArtinN} and \cite[Proposition 1.11, Remark 1.12, and Corollary 1.8 (c)]{U1}, $J$ satisfies $G_\infty$ and sliding depth.  
 Thus by \cite[Theorem 6.1]{HSV_II} we have that ${\rm gr}_{J}(R)=\oplus_{n\gs 0}\,J^n/J^{n+1}$  is Cohen-Macaulay. 
Moreover $\{J^{n+1}:L^{n}\}_{n \in \mathbb{N}}$ forms a decreasing sequence of ideals.  Indeed, since ${\rm gr}_{J}(R)$ is Cohen-Macaulay, we have $J^{n+i}:J^n=J^i$ for every non-negative integers $i$ and $n$. Therefore   
$$ J^{n+1}:L^n=(J^{n+2}:J):L^n=J^{n+2}:JL^n\supseteq J^{n+2}:L^{n+1}.$$ 

Computing $a$-invariants we have $a(A)=a({\rm gr}_J(R))+1$, as ${\rm gr}_{J}(R)\cong A/(t^{-1})$. Furthermore $a({\rm gr}_{J}(R))=-g$ by \cite[Theorem 3.5]{SUV1}. Therefore $[\omega_A]_{g-1}t^{1-g}=R$, which implies $At^{g-1} \subseteq \omega_A$.   Hence by \autoref{core in formula}   for all $0\ls i\ls r$ we have
$$
\core(L) \subseteq J^{r+1}:_RL^{r} \subseteq J^{i+1}:_RL^{i}
=[At^{g-1}]_{g+i}t^{-g-i}:_RL^i\\
\subseteq  [\omega_A]_{g+i}t^{-g-i}:_RL^i.
$$
Thus $\core(L)  \subseteq \bigcap \limits_{i=0}^{r}( [\omega_A]_{g+i}t^{-g-i}:_{R}L^{i})=[\omega_B]_gt^{-g}={\rm adj}(L^g)$ by \autoref{wB} and \autoref{eq_adj}, as desired.
\end{proof}

\begin{corollary}\label{upper}
Let $R=\kk[x_1,\ldots,x_d]$ be a polynomial ring over a field $\kk$ of characteristic zero and $\fm=(x_1,\ldots, x_d)$ the maximal  homogeneous 
 ideal of $R$.  If $L$ is a lex-segment ideal of   height $g\gs 2$ generated in degree $\delta\gs 2$, then $$\core(L)\subseteq L\fm^{(g-1)\delta-d+1}.$$  
\end{corollary}
\begin{proof}
The ideal $L^g$ is integrally closed  by \autoref{LisNormal} and it is generated in a single degree. Hence \cite[Main Theorem]{Howald}  implies that if $\fx^\fv\in \adj(L^g)$  then $\fx^\fv x_1x_2\cdots x_d\in L^g\fm$. Therefore 
$\adj(L^g)\subseteq \fm^{g\delta-d+1}$. 
Finally, by \autoref{core_in_adj}  we conclude that $\core(L)\subseteq L\cap \adj(L^g)\subseteq L\cap \fm^{g\delta-d+1}=L\fm^{(g-1)\delta-d+1}$.
\end{proof}

\section*{Acknowledgments}
The second author was partially funded by NSF Grant DMS \#2001645/2303605. 
  The third author was partially funded by NSF Grant DMS \#2201110. 
  The fourth author was partially funded by NSF Grant DMS \#2201149. 
Part of this research  was done at the Mathematisches Forschungsinstitut Oberwolfach (MFO) while some of the authors were in residence 
during several research visits. 
The authors thank MFO for its hospitality and excellent working conditions.

\end{document}